\documentclass[aap,MSNbibl,seceqn,dvips]{arximspdf}
\usepackage{graphicx}
\doi{10.1214/13-AAP985} 
\volume{24}
\issue{6}
\pubyear{2014}
\firstpage{2560}
\lastpage{2594}

\makeatletter
\newcommand{\widebar}{\overline}

\newcommand{\rrVert}{\Vert}
\newcommand{\rrvert}{\vert}
\newcommand{\llVert}{\Vert}
\newcommand{\llvert}{\vert}

\newcommand{\undersett}[2]{\mathop{#2}\limits_{#1}}
\newtheorem{theoreme}{Theorem}[section]
\newtheorem{prop}[theoreme]{Proposition}
\newtheorem{coro}[theoreme]{Corollary}
\newtheorem{lemme}[theoreme]{Lemma}
\newproclaim{defi}{Definition}[section]
\newtheorem{conjecture}[theoreme]{Conjecture}
\newproclaim{remark}{Remark}[section]
\newcommand{\R}{{\mathbb{R}}}
\newcommand{\N}{{\mathbb{ N}}}
\newcommand{\D}{{\mathbb{D}}}
\newcommand{\E}{{\mathbb{E}}}
\newcommand{\Lm}{{\mathcal{L}}}
\newcommand{\lb}{{\psi}}
\makeatother

\begin{document}
\begin{frontmatter}

\title{The component sizes of a critical random graph with given
degree sequence}
\runtitle{The component sizes of a critical random graph}

\begin{aug}
\author{\fnms{Adrien} \snm{Joseph}\corref{}\ead[label=e1]{adrien.joseph@normalesup.org}}
\runauthor{A. Joseph}
\affiliation{Universit\'{e} Pierre et Marie Curie}
\address{Laboratoire de Probabilit\'{e}s et Mod\'{e}les Al\'
{e}atoires\\
Universit\'{e} Pierre et Marie Curie\\
4 place Jussieu Tour 16-26\\
75005 Paris\\
France\\
\printead{e1}} 
\end{aug}

\received{\smonth{12} \syear{2010}}
\revised{\smonth{8} \syear{2013}}

%
\begin{abstract}
Consider a critical random multigraph $\mathcal{G}_n$ with $n$
vertices constructed by the configuration model such that its vertex
degrees are independent random variables with the same distribution
$\nu$ (criticality means that the second moment of $\nu$ is finite
and equals twice its first moment). We specify the scaling limits of
the ordered sequence of component sizes of $\mathcal{G}_n$ as $n$
tends to infinity in different cases. When $\nu$ has finite third
moment, the components sizes rescaled by $n^{-2/3}$ converge to the
excursion lengths of a Brownian motion with parabolic drift above past
minima, whereas when $\nu$ is a power law distribution with exponent
$\gamma\in(3,4)$, the components sizes rescaled by $n^{-(\gamma
-2)/(\gamma-1)}$ converge to the excursion lengths of a certain
nontrivial drifted process with independent increments above past
minima. We deduce the asymptotic behavior of the component sizes of a
critical random simple graph when $\nu$ has finite third moment.
\end{abstract}

%

\end{frontmatter}

\setcounter{footnote}{0}

\section{Introduction}\label{sec1}

\subsection{Overview}

The classical random graph model $G(n,p)$ has received a lot of
attention since its introduction by Erd\H{o}s and R\'enyi \cite
{MR0125031}, especially because of the existence of a phase transition.
In this model, a graph on $n$ labeled vertices is constructed randomly
by joining any pair of vertices by an edge with probability $p$,
independently of the other pairs. For large $n$, the structure of this
random graph depends on the value of $n p$: for $p \sim c/n$ with $c <
1$, the largest connected component contains $O( \ln n)$ vertices,
whereas when $p \sim c/n$ with $c > 1$, the largest component has
$\Theta(n)$ vertices while the second largest component has $O( \ln
n)$ vertices. The cases $c < 1$ and $c > 1$ are called subcritical and
supercritical, respectively. Much attention has been devoted to the
critical case $p \sim1/n$. When\vspace*{1pt} $p$ is exactly equal to $1/n$, the
largest components of $G(n,p)$ have sizes of order $n^{2/3}$.

Molloy and Reed \cite{MR1370952} showed that a random graph with a
given degree sequence exhibits a similar phase transition. More
precisely, for each $n \geq1$, let $\mathbf{d}^{(n)} =
(d_i^{(n)})_{1 \leq i \leq n}$ be a nonincreasing sequence of positive
integers such that $\sum_{i=1}^n d_i^{(n)}$ is even. Let $G(n,
\mathbf{d}^{(n)})$ be a random simple graph on $n$ labeled
vertices with degree sequence $\mathbf{d}^{(n)}$, uniformly chosen
among all possibilities (tacitly assuming that there exists any such
graph). We suppose throughout the overview that there exists a
probability distribution $(\nu_k)_{k \geq1}$ such that for each $k$,
$\# \{ i\dvtx d_i^{(n)} = k \} / n \rightarrow\nu_k$ as $n \rightarrow
\infty$. Let $\omega(n) = d_1^{(n)}$ be the largest degree in the
graph. Under some further strong conditions on the sequences
$\mathbf{d}^{(n)}$, Molloy and Reed proved that if $Q= \sum
_{k=1}^{\infty} k (k-2) \nu_k < 0$ and $\omega(n) \leq n^{1/8 -
\varepsilon}$ for some $\varepsilon> 0$, then with probability
tending to 1, the size of the largest component of $G(n, \mathbf
{d}^{(n)})$ is $O(\omega^2(n) \ln n)$, whereas if $Q > 0$ and $\omega
(n) \leq n^{1/4 - \varepsilon}$ for some $\varepsilon> 0$, then with
probability tending to 1, the size of the largest component is $\Theta
(n)$, and if additionally $Q$ is finite, the size of the second largest
component is $O(\ln n)$.

More recently, the near-critical behavior of such graphs has been
studied. When $Q=0$, the structure of $G(n, \mathbf{d}^{(n)})$
depends on how fast the quantity
\[
\alpha_n = \sum_{k=1}^{\infty} k
(k-2) \frac{ \# \{ i\dvtx d_i^{(n)} =
k \} }{ n} = \sum_{i=1}^n
\frac{d_i^{(n)} (d_i^{(n)}-2)}{n} %
\]
converges to 0; see Kang and Seierstad \cite{kangsierst}. Requiring a
fourth moment condition, Janson and Luczak \cite{jansonlucanweapp}
proved that if $n^{1/3} \alpha_n \rightarrow\infty$, then the size
of the largest component of $G(n, \mathbf{d}^{(n)})$ divided by $
n \alpha_n$ converges in probability to $\frac{2 \mu}{\beta}$,
while the size of the second largest component of $G(n, \mathbf
{d}^{(n)})$ divided by $ n \alpha_n$ converges in probability to 0,
where $\mu= \sum_{k=1}^{\infty} k \nu_k$ and $\beta= \sum_{k=3}^{\infty
} k(k-1)(k-2) \nu_k \in(0, \infty)$. Furthermore,
they noticed that their results can also be applied to some other
random graph models by conditioning on the vertex degrees, provided
that the random graph conditioned on the degree sequence has a uniform
distribution over all possibilities. This is the case for $G(n, p)$
with $n p \rightarrow1$ and $n^{1/3} (n p - 1) \rightarrow\infty$.
Note that if $n^{1/3} (n p - 1) = O(1)$, it is well known that the
largest component and the second largest component both have sizes of
the same order $n^{2/3}$, so that their results do not hold.

A major difficulty when dealing with the natural random graph $G(n,
\mathbf{d}^{(n)})$ is that, despite its straightforward
definition, it cannot be constructed via an easy algorithm. To
circumvent that obstacle, it is convenient to work with \emph
{multigraphs}, in which multiple edges and loops are allowed, using the
explicit procedure provided by the \emph{configuration model}, which
was introduced by Bender and Canfield~\cite{MR0505796} and later
studied by Bollob{\'a}s \cite{bollobasssss} and Wormald \cite
{wormald}. See also Molloy and Reed \mbox{\cite{MR1370952,MR1664335}}, Kang
and Seierstad \cite{kangsierst}, Bertoin and Sidoravicius \cite
{bertoinbreeeeesil}, van der Hofstad \cite{remcococ} and Hatami and
Molloy \cite{hatamireeeed}. Specifically, take a set of $d_i^{(n)}$
half-edges for the vertex with label $i$, $i \in\{1,\ldots, n \}$,
and combine the half-edges into pairs by a uniformly random matching of
the set of all half-edges. Observing that every simple graph $G(n,
\mathbf{d}^{(n)})$ may be constructed through the same number,
$d_1^{(n)} ! \cdots d_n^{(n)} !$, of pairing of half-edges, we get that
conditional on being a (simple) graph, the multigraph obtained by the
configuration model has the same distribution as $G(n, \mathbf
{d}^{(n)})$. That is why we shall first deal with multigraphs. We shall
then see how to derive results for simple graphs.

\subsection{The present model}

The present work is devoted to studying $G(n,\break  \mathbf{d}^{(n)})$
for a family of degree sequences that are, in a certain sense, ``inside
the critical window.'' We suppose that we are given a probability
distribution $\nu= (\nu_k)_{k \geq1}$ with finite second moment such
that $\nu_2 < 1$ and
$
\sum_{k=1}^{\infty} k (k-2) \nu_k = 0$.
Let $D$ be a random variable with distribution $\nu$. The multigraph
$\mathcal{G}_n$ consisting of $n$ vertices is defined by the
configuration model as follows. Let $D_1, D_2,\ldots, D_n$ be $n$
independent copies of $D$. Condition on $\sum_{i=1}^n D_i$ being even.
Take a set of $D_i$ half-edges for each vertex, and combine the
half-edges into pairs by a uniformly random matching of the set of all
half-edges. We denote by $\mathcal{G}_n$ the random multigraph this
construction leads to.

Let $\bolds{\mathcal{C}}^{\nu}_n$ be the ordered sequence of
component sizes of $\mathcal{G}_n$. We aim at specifying the
asymptotics of $\bolds{\mathcal{C}}^{\nu}_n$ in two different
settings. First, we shall study the case when $\nu$ has finite third
moment. We shall prove that $n^{-2/3} \bolds{\mathcal{C}}^{\nu
}_n$ then converges in distribution (with respect to a certain topology
that will be detailed below) as $n \rightarrow\infty$ to the ordered
sequence of the excursion lengths of a Brownian motion with parabolic
drift; see Theorem~\ref{keytheorr} below for the precise statement.
This should be viewed as an extension of Aldous's well-known result for
the critical behavior of Erd\H{o}s--Renyi random graphs; see \cite
{MR1434128}. Next the case when $\nu$ is a power law distribution with
exponent $\gamma\in(3,4)$ will be studied. We shall show that
$n^{-(\gamma-2)/(\gamma-1)} \bolds{\mathcal{C}}^{\nu}_n$
converges in distribution as $n \rightarrow\infty$ to the ordered
sequence of the excursion lengths of a certain nontrivial drifted
process with independent increments; see Theorem~\ref{levraikeytheorr} below.

Similar results have already been obtained for different random graph
models. For example, Turova \cite{turovvaaa} and Bhamidi, van~der Hofstad and van Leeuwaarden
\mbox{\cite{MR2735378,remcoetcompagnii}} studied special cases of \mbox{rank-1}
inhomogeneous random graphs constructed as follows. Let $F$ be a
distribution function on $[0, \infty)$ and $w_1, w_2,\ldots, w_n$ be
defined by
$
w_i = [1-F]^{-1}(i/n)$.
Consider a simple graph on $n$ labeled vertices such that an edge joins
the vertices $i$ and $j$ ($i \neq j$) with probability $1-\exp(-w_i
w_j/l_n)$, where $l_n = \sum_{i=1}^n w_i$, different edges being
independent. Denoting by $W$ a r.v. with distribution function $F$,
suppose that $\mathbb{E}[W^2] < \infty$. The criticality of the model
occurs when $\mathbb{E}[W^2] = \mathbb{E}[W]$. As in the present
work, two different settings have been considered. In the case $\mathbb
{E}[W^3]<\infty$, Turova \cite{turovvaaa} and Bhamidi, van~der Hofstad and van Leeuwaarden \cite
{MR2735378} separately showed that the ordered sequence of component
sizes of the inhomogeneous random graph with $n$ vertices once rescaled
by $n^{-2/3}$ converges in distribution as $n \rightarrow\infty$ to
the ordered sequence of the excursion lengths of a Brownian motion with
parabolic drift, thus extending the results of Aldous \cite
{MR1434128}. As for the power law distribution case, Bhamidi, van~der Hofstad and van Leeuwaarden
\cite{remcoetcompagnii} proved that if there exist $\gamma\in(3,4)$
and $c >0$ such that $1-F(x) \sim_{x \rightarrow\infty} c
x^{1-\gamma}$, the ordered sequence of component sizes of the
inhomogeneous random graph with $n$ vertices once rescaled by
$n^{-(\gamma-2)/(\gamma-1)}$ then converges in distribution as $n
\rightarrow\infty$ to hitting times of a thinned L\'evy process. This
convergence is related to certain cases of the results obtained by
Aldous and Limic in \cite{MR1491528}.

We too shall be interested in random simple graphs. Specifically, let
$\mathcal{SG}_n$ be the random simple graph consisting of $n$ vertices
such that, conditionally on the degree sequence $(D_1,\ldots, D_n)$,
it is uniformly distributed over all simple graphs with this degree
sequence. Denoting by $\bolds{\mathcal{D}}^{(n)}$ the ordered
sequence of $(D_1,\ldots, D_n)$, $\mathcal{SG}_n$ has the same
distribution as $G(n, \bolds{\mathcal{D}}^{(n)})$. The random
simple graph $\mathcal{SG}_n$ may also be viewed as the multigraph
$\mathcal{G}_n$ conditioned to be simple. When $\nu$ has finite third
moment, we shall be able to prove that the ordered sequence
$\bolds{\mathcal{SC}}^{\nu}_n$ of component sizes of the graph
$\mathcal{SG}_n$ has the same asymptotic behavior as $\bolds{\mathcal
{C}}^{\nu}_n$; see Theorem~\ref{lasttheofmylifeee} below.
We refer to Britton, Deijfen and Martin-L{\"o}f \cite{brittonlof} for an understanding
of the link between inhomogeneous random graphs and $\mathcal{SG}_n$.

The paper is organized as follows. In Sections~\ref{sectionfrretwo},
\ref{deuxpart}, \ref{poivguyiaeuy}, \ref{sectkiuyy}, \ref
{ghyttteerfeiii} and \ref{christinagoldlousub}, we deal with the
finite third moment case. Apart from Section~\ref
{christinagoldlousub}, the main techniques developed there are used in
Section~\ref{powerlawsecr}, where the power law distribution case is
studied. Section~\ref{christinagoldlousub}, devoted to $\bolds
{\mathcal{SC}}^{\nu}_n$, is specific to the finite third moment case.
The main results will be stated in Section~\ref{sectionfrretwo}. In
Section~\ref{deuxpart}, following the ideas of Aldous \cite
{MR1434128}, we shall observe that the study may be reduced to the
understanding of a walk defined via an algorithmic procedure related to
the configuration model. Thanks to \cite{MR1434128}, convergence of
that walk turns out to be sufficient. Such convergence will be obtained
in Section~\ref{sectkiuyy} using standard methodology from stochastic
process theory; see, for example, the CLT for continuous-time
martingale. A key technique to obtain martingales is Poissonization.
Basically, instead of considering multigraphs with exactly $n$
vertices, we shall deal with multigraphs with $\operatorname{Poisson}(n)$ vertices.
This will be fully explained in Section~\ref{poivguyiaeuy}. Our
approach also relies on size-biased ordering. Finally, in Section~\ref
{ghyttteerfeiii}, we shall be interested in the number of cycles in the
multigraph $\mathcal{G}_n$. To conclude, in Section~\ref
{powerlawsecr}, we shall study $\bolds{\mathcal{C}}^{\nu}_n$
when $\nu$ is a power law distribution with exponent in $(3,4)$. We
shall follow the same strategy, except we shall apply results of Aldous
and Limic \cite{MR1491528}. The final \hyperref[app]{Appendix} puts together technical lemmas.

\section{Formulation of the main results in the finite third moment setting} \label{sectionfrretwo}

In the first sections of the paper, we suppose that $\nu$ satisfies
%
%
\begin{equation}
\label{revivalprepaa} \sum_{k=1}^{\infty} k (k-2)
\nu_k =0, \qquad\sum_{k=1}^{\infty}
k^3 \nu_k < \infty\mbox{ and }\nu_2 < 1.
\end{equation}
The more general power law distribution case will be studied in
Section~\ref{powerlawsecr}.

Let
\[
\mu= \sum_{k=1}^{\infty} k \nu_k
\quad\mbox{and}\quad\beta= \sum_{k=3}^{\infty}
k(k-1) (k-2) \nu_k. %
\]
Observe that $\beta> 0$. Define the Brownian motion with parabolic drift
\[
W^{\nu} (t) = \sqrt{\frac{\beta}{\mu}} W(t) - \frac{\beta}{2 \mu
^2}
t^2,\qquad t \geq0, %
\]
where $(W (t), t \geq0)$ is a standard Brownian motion. The reflected
process indexed by the nonnegative half-line is
\[
R^{\nu} (t) = W^{\nu}(t) - \min_{0 \leq s \leq t}
W^{\nu
}(s),\qquad t \geq0. %
\]
An interval $\gamma= [l(\gamma), r(\gamma)]$ is an \emph{excursion
interval} of $R^{\nu}$ if $R^{\nu}(l(\gamma)) = R^{\nu}(r(\gamma))
= 0$ and $R^{\nu}(t) > 0$ on $l(\gamma) < t < r(\gamma)$. The
excursion has length $|\gamma| = r(\gamma) - l(\gamma)$. Aldous
observed in \cite{MR1434128} that we can a.s. order excursions by
length, that is, the set of excursions of $R^\nu$ may be written $\{
\gamma_j, j \geq1 \}$ so that the lengths $|\gamma_j|$ are
decreasing. In the notation of \cite{MR1434128}, define $l_{\searrow
}^2$ as the set of infinite~sequences $x = (x_1, x_2,\ldots)$ with
$x_1 \geq x_2 \geq\cdots\geq0$ and $\sum_i x_i^2 < \infty$, endowed
with the Euclidean metric. Aldous showed in \cite{MR1434128}, Lemma~25,
that $\mathbb{E}[\sum_{j \geq1} |\gamma_j|^2] < \infty$. In
particular $(|\gamma_j|, j \geq1)$ a.s. belongs to $l_{\searrow}^2$.
On the other hand, we may regard the finite sequence $\bolds{\mathcal
{C}}^{\nu}_n$ as a random element of $l_{\searrow}^2$ by
appending zero entries.

Our main result describes the component sizes of $\mathcal{G}_n$ for
large $n$; it mirrors that of Aldous \cite{MR1434128} for the critical
random graph.
%
%
\begin{theoreme} \label{keytheorr}
Suppose $\nu$ satisfies (\ref{revivalprepaa}).
Let $\bolds{\mathcal{C}}^{\nu}_n$ be the ordered sequence of
component sizes of $\mathcal{G}_n$. Then
\begin{eqnarray*}
n^{-2/3} \bolds{\mathcal{C}}^{\nu}_n &
\mathop{\longrightarrow}\limits
_{n \rightarrow\infty}^{(d)} & \bigl( \llvert
\gamma_j \rrvert, j \geq1 \bigr)
\end{eqnarray*}
with respect to the $l_{\searrow}^2$ topology.
\end{theoreme}
We shall observe that Theorem~\ref{keytheorr} is a direct corollary of
a simpler result, namely Theorem~\ref{thetheoaldsboooo}; see the
remark after its statement.

%
%
\begin{remark}
Suppose $\nu_2 = 1$, that is, $D \equiv2$. Then the components of
$\mathcal{G}_n$ are cycles. It is well known that the distribution
of\vadjust{\goodbreak}
cycle lengths is given by the Ewens's sampling formula ESF$(1/2)$, and
thus the size of the largest component divided by $n$ converges in
distribution to a nondegenerate distribution on $[0,1]$; see \cite
{MR2032426}, Lemma~5.7. This is also the case for the $k$th largest
component, where $k$ is a fixed positive integer. That is why the
assumption $\nu_2 < 1$ made in (\ref{revivalprepaa}) is crucial.
\end{remark}

Note that in our setting,
%
%
\begin{equation}
\label{derniequhabyvu} \liminf_{n \rightarrow\infty} \mathbb{P} (\mathcal
{G}_n\mbox{ is a simple graph} ) > 0;
\end{equation}
see Bollob{\'a}s \cite{MR1864966}, Janson \cite{jansonnnnn}. Recall
that $\mathcal{SG}_n$ is the random simple graph such that,
conditioned on the degree sequence $(D_1,\ldots, D_n)$, it is
uniformly distributed over all graphs with this degree sequence. But it
is also the multigraph $\mathcal{G}_n$ conditioned on being simple.
That is why authors usually first focus on $\mathcal{G}_n$ to then
deduce results for $\mathcal{SG}_n$ using (\ref{derniequhabyvu});
see, for instance, Pittel \cite{MR2434184}, Janson \cite
{janspowelaw}, Janson and Luczak \cite{jansonlucanweapp}. In the
finite third moment setting, we shall be able to set up this strategy;
we shall prove an analogous result of Theorem~\ref{keytheorr}:
%
%
\begin{theoreme} \label{lasttheofmylifeee}
Suppose $\nu$ satisfies (\ref{revivalprepaa}).
Let $\bolds{\mathcal{SC}}^{\nu}_n$ be the ordered sequence of
component sizes of $\mathcal{SG}_n$. Then
\begin{eqnarray*}
n^{-2/3} \bolds{\mathcal{SC}}^{\nu}_n &
\mathop{\longrightarrow}\limits
_{n \rightarrow\infty}^{(d)} & \bigl( \llvert
\gamma_j \rrvert, j \geq1 \bigr)
\end{eqnarray*}
with respect to the $l_{\searrow}^2$ topology.
\end{theoreme}

As before, we shall derive Theorem~\ref{lasttheofmylifeee} from a
simpler result stated in Theorem~\ref{jesuisallevoirfghfytu}.

%
%
\begin{remark}
Consider the case when $\nu$ is the Poisson distribution with
parameter 1 [observe though that $\mathbb{P}(D=0) > 0$, so strictly
speaking, it is out of our setting, but our result still holds as
vertices with degree 0 play no role]. Then, for large integers $n$,
$\mathcal{SG}_n$ is an approximation of the Erd\H{o}s--Renyi random
graph $G(n,1/n)$. Now, in that case, $\mu=\beta=1$, so the process
$W^{\nu}$ is the Brownian motion with drift $-t$ at time $t$, which
also describes the asymptotic component sizes of $G(n,1/n)$; see \cite
{MR1434128}.
\end{remark}

\section{The depth-first search} \label{deuxpart}

\subsection{An algorithmic construction of $\mathcal{G}_n$} \label{boulettemsckk}

We start by describing a convenient algorithm to construct a multigraph
distributed as $\mathcal{G}_n$. Suppose that $\sum_{i=1}^n D_i$ is
even. We partition the set of half-edges into three subsets: the set
$\mathcal{S}$ of sleeping half-edges, the set $\mathcal{A}$ of active
half-edges and the set $\mathcal{D}$ of dead half-edges. $\mathcal{S}
\cup\mathcal{A}$ is the set of living half-edges. Initially, all the
half-edges are sleeping.\vadjust{\goodbreak}

Pick a sleeping half-edge uniformly at random, and let $v_1$ denote the
vertex it is attached to. Declare all the half-edges attached to $v_1$
active. While $\mathcal{A} \neq\varnothing$, proceed as follows:
\begin{itemize}
\item Let $i$ be the largest integer $k$ such that there exists an
active half-edge attached to $v_k$.
\item Consider an active half-edge $l$ attached to $v_i$.
\item Kill $l$, that is, remove it from $\mathcal{A}$, and place it
into $\mathcal{D}$.
\item Choose uniformly at random a living half-edge $r$ and pair $l$ to it.
\item If $r$ is sleeping, let $v_{j+1}$ denote the vertex it is
attached to, where $j$ is the number of vertices which were found
before the discovery of the vertex attached to $r$. Then declare all
the half-edges attached to $v_{j+1}$ except $r$ active.
\item Kill $r$.
\end{itemize}
Iterate until $\mathcal{A}=\varnothing$. At that step, the first
component has been totally explored. If $\mathcal{S} \neq\varnothing$,
proceed similarly with the remaining living vertices until all the
half-edges have been killed. Then consider the multigraph with vertex
set $\{ v_i, 1 \leq i \leq n \} $ such that for all $1 \leq i,j \leq
n$, the vertex $v_i$ is joined by $k$ edges to the vertex $v_j$ if and
only if $k$ half-edges of $v_i$ have been paired to $k$ other
half-edges of $v_j$ during the procedure. It is easily seen this
multigraph is distributed as $\mathcal{G}_n$ and its vertices have
been ordered via a depth-first search. See Figure~\ref{figureunn} above for a simple illustration.

%
%
\begin{figure}

\includegraphics{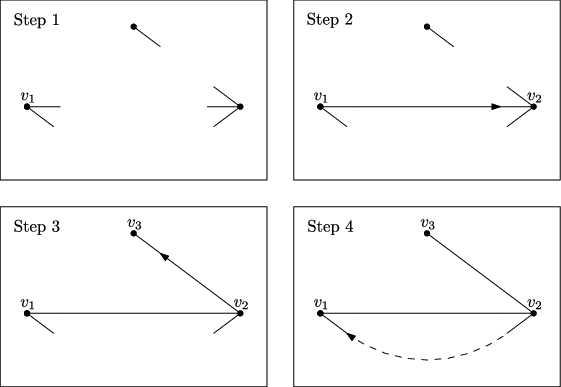}

\caption{A realization of the algorithm constructing $\mathcal{G}_3$.
The dashed oriented edge of the last picture contains a cycle half-edge
at its origin: $v_2$ has a cycle half-edge. By definition, $W_3(0)=0$,
$W_3(1)=0$, $W_3(2)=-1$ and $W_3(3)=-2$.}
\label{figureunn}
\end{figure}

Also note that, by construction, the order in which the components
appear in the depth-first search is size-biased order.

\subsection{The depth-first walk}

We now explain how the information on the component sizes may be
encoded in a walk constructed via the depth-first search which, as we
shall see, is related to the process $W^{\nu}$. We first need the
notion of \emph{cycle half-edge}.

%
%
\begin{defi}
A half-edge $l$ is called a \emph{cycle half-edge} if there exists a
half-edge $r$ such that:
\begin{itemize}
\item$l$ was killed before $r$;
\item$l$ was paired to $r$;
\item$r$ was active when $l$ was paired to it.
\end{itemize}
\end{defi}
%

Let us now define the walk associated to the depth-first search which
will encode all the information that we need to study the component
sizes. Write $(\widehat{D}_i, i \in\{1, 2,\ldots, n \})$ the sequence of
the degrees of the vertices of $\mathcal{G}_n$ ordered by their
appearances in the depth-first search:
for every $i \in\{1,\ldots, n \}$,
\[
\widehat{D}_i = \mbox{degree of }v_i. %
\]
Define the depth-first walk $(W_n(i), 0 \leq i \leq n)$ by letting for
all $i \in\{0,\ldots, n \}$,
%
%
\begin{eqnarray}
\label{defindelamarcle} W_n(i) & = & \sum_{j = 1}^i
\bigl( \widehat{D}_j -2 - 2 \# \{\mbox{cycle half-edges attached to
$v_j$} \} \bigr).
\end{eqnarray}
Note that since the cycle half-edges attached to $v_j$ always appear
after $v_j$ has been discovered, the number of them is not measurable
with respect to the first $j$ steps of the process.

Order the components $\mathcal{C}(n,1), \mathcal{C}(n,2),\ldots$
according to the depth-first search. Let
\begin{eqnarray*}
\zeta(n,k) & = & \sum_{j=1}^{k} \bigl|
\mathcal{C}(n,j)\bigr|,
\\
\zeta^{-1}(n,i) & = & \min\bigl\{ k\dvtx \zeta(n,k) \geq i \bigr\},
\end{eqnarray*}
so that $\zeta^{-1}(n,i)$ is the index of the component containing $v_i$.
It is easily seen that
%
%
\begin{eqnarray}\label{cmechantmsc}
W_n\bigl(\zeta(n,k)\bigr) = -2k\quad\mbox{and}\quad W_n(i) \geq-2k -1
\nonumber\\[-9pt]\\[-9pt]
\eqntext{\mbox{for all } \zeta(n,k) \leq i <\zeta(n,k+1).}
\end{eqnarray}
It follows that we can recover component sizes and indices from the
walk via
\begin{eqnarray*}
\zeta(n,k) & = & \min\bigl\{ i\dvtx W_n(i) = -2k \bigr\},
\\
\bigl|\mathcal{C}(n,j)\bigr| & = & \zeta(n,j)- \zeta(n,j-1),
\\
\zeta^{-1}(n,i) & = & 1 - \biggl\lceil\min_{j < i}
\frac{W_n(j)}{2} \biggr\rceil.
\end{eqnarray*}

\subsection{Weak convergence on every finite interval of the depth-first walk}

Let $X_n$, $n \geq1$, and $X$ be $\R$-valued C\`adl\`ag processes
defined on $[0, \infty)$. For every $t > 0$, denote by $\D([0,t])$
the space of all $\R$-valued c\`adl\`ag functions defined on $[0,t]$
endowed with the Skorokhod topology. Throughout this work, we say that
$X_n$ converges in distribution to $X$ with respect to the Skorokhod
topology on every finite interval as $n \rightarrow\infty$ if for
every $t >0$ and every bounded, continuous function $f$ defined on $(\D
([0,t]), \mathbb{R})$,
\[
\E\bigl[ f (X_n ) \bigr] \mathop{\longrightarrow}_{n \rightarrow\infty}  \E\bigl[f (X ) \bigr] %
\]
(here, we write $X_n$ and $X$ for their restrictions to the interval $[0,t]$).

Our main result relates the walk to the process $W^{\nu}$:
%
%
\begin{theoreme} \label{thetheoaldsboooo}
Suppose $\nu$ satisfies (\ref{revivalprepaa}). Rescale the
depth-first walk $W_n$ by defining for every $t \in[0, n^{1/3}]$
\begin{eqnarray*}
\widebar{W}_n(t) & = & n^{-1/3} W_n \bigl(
\bigl\lfloor t n^{2/3} \bigr\rfloor\bigr).
\end{eqnarray*}
Then
\begin{eqnarray*}
\widebar{W}_n & \mathop{\longrightarrow}\limits
_{n \rightarrow
\infty}^{(d)} &
W^{\nu}
\end{eqnarray*}
with respect to the Skorokhod topology on every finite interval.
\end{theoreme}

To see how Theorem~\ref{keytheorr} follows from Theorem~\ref
{thetheoaldsboooo}, we refer to Section~3.4 of the remarkable paper
\cite{MR1434128} of Aldous.\footnote{Recall that components appeared
in size-biased order in the depth-first walk.} Intuitively, the result
should be clear from property (\ref{cmechantmsc}) of depth-first walk.
Component sizes are indeed encoded as lengths of path segments above
past even minima; these converge to lengths of excursions of $W^{\nu}$
above past minima, which are just lengths of excursions of the
reflected process $(W^{\nu}(t) - \min_{0 \leq s \leq t} W^{\nu}(s),
t \geq0)$ above 0.
Similarly, Theorem~\ref{lasttheofmylifeee} is proven as soon as the
following result is shown:
%
%
\begin{theoreme} \label{jesuisallevoirfghfytu}
If $\nu$ satisfies (\ref{revivalprepaa}), then the rescaled walk
$\widebar{W}_n$ conditioned on the event $\{ \mathcal
{G}_n$ is simple$\}$ converges in distribution to $W^{\nu}$ with
respect to the Skorokhod topology on every finite interval as $n
\rightarrow\infty$.
\end{theoreme}

The next three sections are devoted to the proof of Theorem~\ref
{thetheoaldsboooo}. Section~\ref{poivguyiaeuy} will introduce the
method. In Section~\ref{sectkiuyy}, we shall be interested in the
depth-first walk $(\sum_{j = 1}^i ( \widehat{D}_j -2), 0 \leq i \leq n)$.
It is easier to study the latter than the walk $W_n$ since\vspace*{1pt} it ignores
cycle half-edges, and its law only depends on the sequence $(\widehat
{D}_j, 1 \leq j \leq n)$, which has the law of the size-biased ordering
of $n$ independent copies of $D$. Let
\[
\bar{s}_n (t) = n^{-1/3} \sum
_{1 \leq j \leq tn^{2/3} } ( \widehat{D}_j -2 ),\qquad t \in\bigl[ 0,
n^{1/3} \bigr]. %
\]
We shall show that the walk $\bar{s}_n$ converges in distribution
to $W^{\nu}$ as $n \rightarrow\infty$. In Section~\ref
{ghyttteerfeiii}, we shall see that the difference between the two
rescaled depth-first walks $\widebar{W}_n$ and $\bar{s}_n$ is
so small that in the limit, these processes have the same behavior. The
combination of the two remarks yields Theorem~\ref{thetheoaldsboooo}.
As for Theorem~\ref{jesuisallevoirfghfytu}, it will be proved in
Section~\ref{christinagoldlousub}.

\section{Poissonization} \label{poivguyiaeuy}

As mentioned above, in this section, we forget the contribution of the
cycle half-edges to the depth-first walk $W_n$ (we shall see in
Section~\ref{ghyttteerfeiii} that there are indeed few cycle
half-edges up to time $t n^{2/3}$ for every fixed $t>0$), and we only
focus on the simpler walk $(\sum_{j = 1}^i ( \widehat{D}_j -2), 0 \leq i
\leq n)$.

It is easily seen that the configuration model defining $\mathcal
{G}_n$ induces a degree-biased ordering of its vertices: conditionally
on the degrees $D_1,\ldots, D_n$, the sequence $(\widehat{D}_1,\ldots,
\widehat{D}_n)$ has the law of a size-biased reordering of the real
numbers $D_1,\ldots, D_n$. Conditionally on $D_1 = d_1,\ldots, D_n=
d_n$, a convenient way to order the vertices of $\mathcal{G}_n$ in a
degree-biased fashion is to assign an exponential clock with parameter
$d_i$ to the vertex $i$, $i \in\{1, 2,\ldots, n\}$, and to order the
vertices according to the times the clocks they are attached to ring.

\subsection{Heuristics}

We are able to sample at the same time both the degrees of the vertices
of $\mathcal{G}_n$ and their reordering in a size-biased way via a
clever Point point process. The only drawback of this approach is that
the total number of vertices of the obtained multigraph is not exactly
$n$ but is a Poisson variable with parameter $n$ (so that to actually
obtain $\mathcal{G}_n$, one has to condition the total number of
vertices to be equal to $n$). Let us be more precise.

Consider a Poisson point process $\Pi_n^{(0)}$ on $\N^* = \{1, 2,\ldots
\}$ with parameter $n \nu$. The total number of its atoms is a
Poisson variable with parameter $n$, and conditionally on this number,
the atoms of $\Pi_n^{(0)}$ are i.i.d. with distribution $\nu$.
Assigning to each of them an exponential clock with appropriate
parameter would order them in a size-biased fashion.

We could have done those two operations directly by defining more
carefully the Poisson point process; indeed define $\Pi_n^{(1)}$ as a
Poisson point process on $(0,\infty) \times\N^*$ with intensity $\pi
_n^{(1)}$, where
\begin{eqnarray*}
\pi_n^{(1)} (\mathrm{d} t, k) & = & n \nu_k k
e^{-kt}\, \mathrm{d} t.
\end{eqnarray*}
Sort the atoms of $\Pi_n^{(1)}$ in increasing order of their $t$-components:
\begin{eqnarray*}
\Pi_n^{(1)} & = & \bigl\{ \bigl(t^{(1)}_1,
k^{(1)}_1 \bigr),\ldots, \bigl(t^{(1)}_{N^{(1)}},
k^{(1)}_{N^{(1)}} \bigr) \bigr\}
\end{eqnarray*}
(we drop the dependency on $n$ in the notations of $t^{(1)}_i$,
$k^{(1)}_i$ and $N^{(1)}$). Then
\begin{eqnarray*}
\bigl\{k^{(1)}_1,\ldots, k^{(1)}_{N^{(1)}}
\bigr\} & \stackrel{(d)} {=} & \Pi_n^{(0)}.
\end{eqnarray*}
Moreover, since $t^{(1)}_i$ corresponds to the exponential clock of
$k^{(1)}_i$, the sequence $(k^{(1)}_1,\ldots, k^{(1)}_{N^{(1)}})$ has\vspace*{1pt}
the law as the size-biased reordering of the real numbers
$k^{(1)}_1,\ldots, k^{(1)}_{N^{(1)}}$. Consequently, conditionally on
${N^{(1)}}=m$, $(k^{(1)}_1,\ldots, k^{(1)}_m)$ has the same
distribution as the random vector $(\widehat{D}_1,\ldots, \widehat{D}_m)$.

As mentioned in the introduction of this section, we are interested in
the walk $(\sum_{j = 1}^i ( \widehat{D}_j -2), 0 \leq i \leq n)$, which
may be viewed as a function having discontinuities at integer-valued
times. We see that we cannot reasonably approximate this function by
$(\sum_{(s,k) \in\Pi^{(1)}_n} (k-2) \mathbf{1}_{s \leq t}, t \geq
0)$, which has discontinuities at $t^{(1)}_1,\ldots,
t^{(1)}_{N^{(1)}}$; the sequence $(t^{(1)}_1, t^{(1)}_2-t^{(1)}_1,\ldots, t^{(1)}_{N^{(1)}}-t^{(1)}_{{N^{(1)}}-1})$ has no chance to look
like $(1,\ldots, 1)$, partly because for every $i \in\{1,\ldots,
{N^{(1)}}-1\}$, $t^{(1)}_{i}-t^{(1)}_{i-1}$ is stochastically dominated
by $t^{(1)}_{i+1} - t^{(1)}_i$ (with convention $t^{(1)}_0 = 0$). We
should thus transform the $t$-components of the atoms: an atom $(t,k)$
should be replaced by $(\phi_n(t),k)$, where $\phi_n$ is an concave
increasing function, so that conditionally on ${N^{(1)}}$, $\phi
_n(t^{(1)}_1), \phi_n(t^{(1)}_2)-\phi_n(t^{(1)}_1),\ldots, \phi
_n(t^{(1)}_{N^{(1)}})-\phi_n(t^{(1)}_{{N^{(1)}}-1})$ are i.i.d. and
close to 1. We shall show in the next section that there exists such a
function $\phi_n$ and that, conditionally on ${N^{(1)}} \geq i$, $\phi
_n(t^{(1)}_i)-\phi_n(t^{(1)}_{i-1})$ is an exponential variable with
parameter 1; see Lemma~\ref{chaosprope} below.

\subsection{Toward the definition of \texorpdfstring{$\Pi_n$}{Pi n}}

It turns out that the function $\phi_n$ is $n(1 - \Lm)$, where $\Lm$
is the Laplace transform of $\nu$,
\[
\Lm(t) = \sum_{k \in\N^{\ast}} e^{-kt}
\nu_k,\qquad t \geq0. %
\]
Indeed, write $\psi$ for the inverse of $1 - \Lm$ and consider a
Poisson point process $\Pi_n$ on $ (0, n) \times\mathbb{N}^{\ast}$
with intensity $\pi_n$, where
\begin{eqnarray*}
\pi_n(\mathrm{d} t, k) & = & \nu_k k e^{-k\psi(t/n)}
\psi'(t/n) \,\mathrm{d} t.
\end{eqnarray*}
Recall that the $k$-components of the atoms of $\Pi_n$ should be
viewed as degrees whereas the $t$-components should be seen as time.
[Note that $\Pi_n$ could have been defined as $\Pi_n = \{ (\tilde{t}_1,
\tilde{k}_1 ),\ldots, (\tilde{t}_N, \tilde{k}_N ) \}$, where\vspace*{1pt}
$N$ is a Poisson variable with parameter $n$ and $(\tilde{t}_i, \tilde
{k}_i)_{i \geq1}$ is a sequence of i.i.d. r.v. with distribution
$\frac{\pi_n}{n}$ \emph{independent} of $N$.] Sort the atoms of $\Pi
_n$ in increasing order of their $t$-components,
\begin{eqnarray*}
\Pi_n & = & \bigl\{ (t_1, k_1),\ldots,
(t_N, k_N) \bigr\}
\end{eqnarray*}
(here again, we drop the dependency on $n$ in the notations). Then, by
standard properties of Poisson point processes,
\begin{eqnarray*}
\bigl( (t_1, k_1),\ldots, (t_N,
k_N) \bigr) & \stackrel{(d)} {=} & \bigl( \bigl(\phi_n
\bigl(t^{(1)}_1 \bigr), k^{(1)}_1 \bigr),\ldots, \bigl(\phi_n \bigl
(t^{(1)}_{N^{(1)}} \bigr),
k^{(1)}_{N^{(1)}} \bigr) \bigr),
\end{eqnarray*}
where $\phi_n = n(1-\Lm) = n \psi^{-1}$. In particluar $\{k_1,\ldots,
k_N\}$ has the same distribution as $\Pi_n^{(0)}$.

As before, conditionally on $N=m$, $(k_1,\ldots, k_m)$ has the same
law as the random vector $(\widehat{D}_1,\ldots, \widehat{D}_m)$. This in
particular holds for $m=n$. Since $N$ is a Poisson variable with
parameter $n$ and, as we shall soon see, we are only interested in what
happens up to time $O(n^{2/3})$, we shall study the process $\Pi_n$
without the latter conditioning. We thus get a Markovian process. Let
us prove that conditionally on $N \geq i$, $t_i-t_{i-1}$ is an
exponential variable with parameter 1.

%
%
\begin{lemme} \label{chaosprope}
The point process $\{t_1,\ldots, t_N\}$ is Poisson point process on
$(0,n)$ with intensity $\mathrm{d} t$.
\end{lemme}

\begin{pf}
Define $p$ as the projection $p\dvtx (t,k) \mapsto t$. Then, by standard
properties of Poisson point processes, $\{t_1,\ldots, t_N\} = p(\Pi
_n)$ is a Poisson point process on $(0,n)$ with intensity $\pi
_n^{(p)}$ characterized by the following:
\[
\mbox{for every Borel subset $A$ of }(0,n),
\qquad\pi_n^{(p)}(A)= \pi_n \bigl(p^{-1} (A) \bigr). %
\]
Therefore, for every Borel subset $A$ of $(0,n)$,
\begin{eqnarray*}
\pi_n^{(p)}(A) &=& \int_A \sum
_{k \in\mathbb{N}^\ast} \nu_k k e^{-k\psi(t/n)}
\psi'(t/n) \,\mathrm{d} t
\\
&=& \int_A \bigl(
\psi^{-1}\bigr)'\bigl(\psi(t/n)\bigr)
\psi'(t/n) \,\mathrm{d} t
= \int_A \mathrm{d}t, %
\end{eqnarray*}
which proves the result.
\end{pf}

\subsection{Keys points of the section}

Let us sum up the points that will be used in the sequel.
%
%
\begin{prop} \label{cbeaulaviiii}
Let $n$ be a positive integer, $(e_i)_{i \geq1}$ be a sequence of
independent exponential variables with parameter 1 and $(U_i)_{i \geq
1}$ be a sequence of \mbox{independent} random variables uniformly distributed
on $(0,1)$ \emph{independent} of $(e_i)_{i \geq1}$. Define
\begin{eqnarray*}
N &=& \max\Biggl\{ i \geq0\dvtx \sum_{j=1}^i
e_j < n \Biggr\},
\\
t_i & =& \sum_{j=1}^i
e_j, \qquad i \geq1,
\\
k_i &= &\sum_{j \in\N^*} j
\mathbf{1}_{d_{i,j-1}< U_i < d_{i,j}}, \qquad1 \leq i \leq N,
\\
\Pi_n & = & \bigl\{ (t_1, k_1),\ldots,
(t_N, k_N) \bigr\},
\end{eqnarray*}
where for every integer $i$ and $j$ such that $1 \leq i \leq N$ and $j
\geq1$,
$
d_{i,j} = \sum_{l=1}^j \nu_l l e^{-l\psi(t_i/n)} \psi'(t_i/n)
$.
Then:
\begin{itemize}
\item$\Pi_n$ is a Poisson point process on $(0,n) \times\N^*$ with
intensity $\pi_n$, where
\begin{eqnarray*}
\pi_n(\mathrm{d} t, k) & = & \nu_k k e^{-k\psi(t/n)}
\psi'(t/n) \,\mathrm{d} t.
\end{eqnarray*}
\item For\vspace*{1pt} every positive integer $m$, conditionally on $N=m$,
$(k_1,\ldots, k_m)$ has the same law as the random vector $(\widehat
{D}_1,\ldots, \widehat{D}_m)$.
\end{itemize}
\end{prop}

In the sequel, $N$, $(t_i)_{i \geq1}$, $(k_i)_{1 \leq i \leq N}$ and
$\Pi_n$ will always refer to those just-defined quantities.

\section{Convergence of the walk \texorpdfstring{$\bar{s}_n$}{s n}} \label{sectkiuyy}

It should now be natural to introduce the process $(S_n(t))_{t \geq0}$
defined as the sum of the $k$-components of the atoms of $\Pi_n$ minus
2 with $t$-components less than or equal to $t$:
\[
S_n(t)= \sum_{(s,k) \in\Pi_n} (k-2 )
\mathbf{1}_{s
\leq t} = \sum_{1 \leq j \leq N}
(k_j-2 ) \mathbf{1}_{t_j \leq t}. %
\]
We can now state the key result of the present work:
%
%
\begin{prop} \label{gilfoallalllavernewww}
Rescale $S_n$ by defining
$
\widebar{S}_n(t) = n^{-1/3} S_n( t n^{2/3} ).
$
Then
\begin{eqnarray*}
\widebar{S}_n & \mathop{\longrightarrow}\limits
_{n \rightarrow
\infty}^{(d)} &
W^{\nu}
\end{eqnarray*}
with respect to the Skorokhod topology on every finite interval.
\end{prop}
\begin{pf}
We follow the ideas of Aldous \cite{MR1434128}. Let
\[
A_n(t) = \int_{(0, n) \times\mathbb{N}^{\ast} } \pi_n (\mathrm{d}s, k) (k-2) \mathbf{1}_{s \leq t},\qquad t \geq0, %
\]
be the continuous bounded variation process such that
\[
M_n (t) = S_n (t) - A_n (t),\qquad t \geq0,
\]
is a martingale. Observe that $A_n$ is deterministic. Just as we
rescaled $S_n$ to form~$\widebar{S}_n$, write $\widebar{A}_n$ and
$\widebar{M}_n$ for the correspondingly rescaled versions of $A_n$
and $M_n$. Proposition~\ref{gilfoallalllavernewww} is shown as soon
the following two results are established:
\[
\forall t_0 >0
\qquad\lim_{n \rightarrow\infty} \sup
_{t \leq t_0} \biggl\llvert{\widebar{A}}_n(t) +
\frac{\beta}{2 \mu^2} t^2 \biggr\rrvert= 0 %
\]
and
\[
\widebar{M}_n \mathop{\longrightarrow}\limits
_{n \rightarrow
\infty}^{(d)}
\sqrt{\frac{\beta}{\mu}} B %
\]
with respect to the Skorokhod topology on every finite interval,
where $B$ denotes a standard Brownian motion. We postpone\vadjust{\goodbreak} their proofs
to the \hyperref[app]{Appendix}; see Lemmas~\ref{babyonemofgt} and \ref{babyonemofgtjju}. The following estimate on $\Lm$:
%
%
\begin{equation}
\label{equdebertoiiii} \mathcal{L}''(t) = \E
\bigl[D^2 \bigr] - \E\bigl[D^3 \bigr] t + \mathop{o}_{t
\rightarrow0}(t)
\end{equation}
is a key ingredient in the proofs.
\end{pf}

We now give a key consequence of Proposition~\ref
{gilfoallalllavernewww} concerning the depth-first walk $\bar{s}_n$.

%
%
\begin{coro} \label{jeanferr}
The rescaled depth-first walk $\bar{s}_n$
converges in distribution to $W^{\nu}$ with respect to the Skorokhod
topology on every finite interval as $n \rightarrow\infty$.
\end{coro}
\begin{pf}
Denote by $\widetilde{S}_n$ the process
\[
\widebar{S}_n \bigl( n^{-2/3} t_{\lfloor u n^{2/3} \rfloor} \bigr),\qquad u \geq0. %
\]
Applying Propositions~\ref{cbeaulaviiii} and \ref{gilfoallalllavernewww}, $\widetilde{S}_n$ converges in distribution to
$W^{\nu}$ with respect to the Skorokhod topology on every finite
interval as $n \rightarrow\infty$. Now, for every $u \geq0$,
\begin{eqnarray*}
\widetilde{S}_n(u) & = & n^{-1/3} \sum
_{1 \leq j \leq u n^{2/3}} ( k_j -2 ).
\end{eqnarray*}
Since conditionally on $N=n$, $(k_1,\ldots, k_n)$ has the same law as
the random vector $(\widehat{D}_1,\ldots, \widehat{D}_n)$ (see
Proposition~\ref{cbeaulaviiii}), we get that for every $u>0$ and every
bounded, continuous function $f$ defined on $(\D([0,u]), \mathbb{R})$,
\begin{eqnarray*}
\E\bigl[ f (\widetilde{S}_n ) \rrvert N=n \bigr] &=& \E\bigl[f
(\bar{s}_n ) \bigr]
\end{eqnarray*}
(here, we write $\widetilde{S}_n$ and $\bar{s}_n$ for their
restrictions to the interval $[0,u]$). We thus just need to see why
\begin{eqnarray*}
\E\bigl[ f (\widetilde{S}_n ) \rrvert N=n \bigr] - \E\bigl[ f
(\widetilde{S}_n ) \bigr] & \undersett{n \rightarrow\infty} {
\longrightarrow} & 0.
\end{eqnarray*}
Now, conditionally on $N=n$, by Proposition~\ref{cbeaulaviiii}, the
sequence $(t_1,\ldots,\break  t_{\lfloor u n^{2/3}\rfloor})$ has the same
distribution as $(n V_1, n V_2,\ldots, n V_{\lfloor u n^{2/3}\rfloor
})$, where $0<V_1< \cdots< V_n<1$ is the ordered statistics of the
family of $n$ i.i.d. variables uniformly distributed on $(0,1)$. In
other words, the distribution of the random vector $(t_1,\ldots,
t_{\lfloor u n^{2/3}\rfloor})$ under\vspace*{1pt} the event $\{N=n\}$ is exactly
the distribution (without conditioning) of
\[
\frac{n}{t_{n+1}} (t_1, t_2,\ldots, t_{\lfloor u
n^{2/3}\rfloor} )
\]
(moreover, the latter random vector is independent of $t_{n+1}$).
Applying Proposition~\ref{cbeaulaviiii}, we thus deduce that
the conditional distribution of $\widetilde{S}_n$ under $\{N=n\}$ is
asymptotically close to the distribution of $\widetilde{S}_n$. We get the
result by applying the dominated convergence theorem (recall that $f$
is bounded and continuous).
\end{pf}

\section{Study of the cycle half-edges} \label{ghyttteerfeiii}

In this section, we turn our attention to the cycle half-edges. In
Section~\ref{hibergggju}, we shall prove that there are few cycle
half-edges in $\mathcal{G}_n$; see Lemma~\ref{formybertoi} below. We
shall then show in Section~\ref{ghyttteerfe} how to derive
Theorem~\ref{thetheoaldsboooo} from Corollary~\ref{jeanferr} and
Lemma~\ref{formybertoi}.

\subsection{Upper bound of the number of cycle half-edges} \label{hibergggju}

In this section, we prove the following result:
%
%
\begin{lemme} \label{formybertoi}
Let $t >0$ and $M > 0$. Introduce the event
\[
E_n(t, M) = \Bigl\{ \max_{i \leq t }  \Bigl\{ \bar{s}_n(i) -
\min_{k \leq i}  \bar{s}_n(k)
\Bigr\} \leq M \Bigr\}. %
\]
Then we have
\[
\limsup_{n \rightarrow\infty} \E\bigl[ \# \bigl\{ \mbox{cycle half-edges
attached to $v_i$, $i \leq t n^{2/3}$, in $
\mathcal{G}_n$} \bigr\} \mathbf{1}_{ E_n(t, M) } \bigr] < \infty.
\]
\end{lemme}
\begin{pf}
We first study the number of active half-edges, given they contribute
to the appearance of cycle half-edges.

We claim that when a half-edge of $v_i$ is about to be paired (in the
algorithmic construction of $\mathcal{G}_n$ described in Section~\ref
{boulettemsckk}), the number $ \# \mathcal{A}$ of active half-edges is
less than or equal to $2+s_n(i)-\min_{j \leq i} s_n(j)$, where $s_n$
denotes the walk $(\sum_{j \leq i} (\widehat{D}_j - 2), 0 \leq i \leq
n)$. To see why this claim is true, first notice that it suffices to
prove it only for the first component. Then observe that the claim is
true if $v_i$ has just been discovered (this can be shown by induction:
this is true for the first vertex~$v_1$ and, when $i>1$, the number of
active half-edges that the discovery of $v_i$ creates is $-1+\mbox{degree of }v_i-1$, which is exactly the increment of $s_n$). On the
other hand, if $v_i$ had already been discovered before, the number of
active half-edges present when a new half-edge of $v_i$ is about to be
paired is less than the number of active half-edges present when the
first half-edge of $v_i$ was about to be paired (due to our choice of
the depth-first search; we go back to $v_i$ only when the vertices
$v_j$, $j>i$, have all been fully explored). As seen in the first
alternative, that last number is at most $2+s_n(i)-\min_{j \leq i}
s_n(j)$. This completes the proof of the claim.

Consequently, under the event $E_n(t, M)$, during the first $\lfloor t
n^{2/3} \rfloor$ steps, $ \# \mathcal{A}$ is always less than or
equal to $2+ M n^{1/3}$.

For every deterministic sequence $(x_1,\ldots, x_n)$ of positive
integers such that $\sum_{i=1}^n x_i$ is even, conditionally on the
event $(\widehat{D}_1,\ldots, \widehat{D}_n)=(x_1,\ldots, x_n)$, one has
\begin{eqnarray*}
&& \E\bigl[ \# \bigl\{\mbox{cycle half-edges attached to $v_i$,
$i \leq t n^{2/3}$} \bigr\} \mathbf{1}_{ E_n(t, M) } |
\widehat{D}_1 = x_1,\ldots, \widehat{D}_n =
x_n \bigr]
\\
&&\qquad = \E\Biggl[ \sum_{i=1}^{ t n^{2/3} }
\sum_{k=1}^{\widehat
{D}_i} \mathbf{1}_{\{ \mathrm{the}\ k\mathrm{th}\ \mathrm{half\mbox{-}edge\ of}\ v_i\
\mathrm{is\ a\ cycle\ half\mbox{-}edge} \}}
\mathbf{1}_{ E_n(t, M) } \biggr\rrvert
\\
&&\hspace*{184pt}\widehat{D}_1 = x_1,\ldots,\widehat{D}_n = x_n \Biggr]
\\
&&\qquad \leq \sum_{i=1}^{ t n^{2/3} } \sum
_{k=1}^{x_i} \mathbb{P} \bigl( \mbox{the $k$th
half-edge of $v_i$ is a cycle half-edge} \llvert
\\
&&\hspace*{107pt}
\widehat{D}_1 = x_1,\ldots, \widehat{D}_n =
x_n\mbox{ and } E_n(t, M) \bigr)
\\
&&\qquad \leq\sum_{i=1}^{ t n^{2/3} } x_i
\frac{ M n^{1/3} +2}{\sum_{m=1}^n x_m - \sum_{m=1}^{ t n^{2/3} } x_m}
\\
&&\qquad  \leq \sum_{i=1}^{ t n^{2/3} } x_i
\frac{ M n^{1/3} + 2}{n-t n^{2/3}}.
\end{eqnarray*}
Consequently,
\begin{eqnarray*}
&& \E\bigl[ \# \bigl\{ \mbox{cycle half-edges attached to $v_i$,
$i \leq t n^{2/3}$, in $\mathcal{G}_n$} \bigr\}
\mathbf{1}_{ E_n(t,
M) } \bigr]
\\
&&\qquad \leq \frac{ M n^{1/3} + 2}{n-t n^{2/3}} \E\Biggl[\sum_{i=1}^{ t
n^{2/3}}
\widehat{D}_i \Biggr]
\\
&&\qquad \leq \frac{ M n^{1/3} + 2}{n-t n^{2/3}} t n^{2/3} \E[\widehat{D}_1 ].
\end{eqnarray*}
Note that
$
\mathbb{E} [ \widehat{D}_1 ] \leq\sum_{k=1}^{\infty} k \frac{k \nu
_k}{\mu}$.
Hence
\[
\limsup_{n \rightarrow\infty} \E\bigl[ \# \bigl\{ \mbox{cycle half-edges
attached to $v_i$, $i \leq t n^{2/3}$} \bigr\} \mathbf
{1}_{ E_n(t, M) } \bigr] \leq2 M t, %
\]
which completes the proof of Lemma~\ref{formybertoi}.
\end{pf}

%
%
\begin{remark}
We can prove that in fact, for every $t > 0$,
\[
\limsup_{n \rightarrow\infty} \E\bigl[ \# \bigl\{ \mbox{cycle half-edges
attached to $v_i$, $i \leq t n^{2/3}$, in $
\mathcal{G}_n$} \bigr\} \bigr] < \infty. %
\]
\end{remark}

%
%
\begin{remark}
We stress that a consequence of \cite{bertoinbreeeeesil}, Theorem~1, is
that the expected total number of half-edges present in a component
containing a cycle half-edge is $o(n)$. This also holds for the
subcritical regime.
\end{remark}

\subsection{End of the proof of Theorem~\texorpdfstring{\protect\ref{thetheoaldsboooo}}{3.1}} \label{ghyttteerfe}

In this section, we prove Theorem~\ref{thetheoaldsboooo}. We keep the
notation of Section~\ref{hibergggju}. Let $t > 0$. Applying
Corollary~\ref{jeanferr} and the Portmanteau theorem, it suffices to
prove that for every bounded, Lipschitz function $f$ defined on $(\D
([0,t]), \mathbb{R})$, $\E[f(\widebar{W}_n)] - \E[f(\bar{s}_n)]$ tends
to 0 as $n \rightarrow\infty$. Let $f$ be such a
function. There exists $K>0$ such that for every $w, w' \in(\D
([0,t]), \mathbb{R})$, $|f(w)| \leq K$ and $|f(w) - f(w')| \leq K \| w
- w' \|$. Let $M>0$. One has
\begin{eqnarray*}
& & \bigl\llvert\E\bigl[f (\widebar{W}_n ) \bigr] - \E\bigl[f (
\bar{s}_n ) \bigr] \bigr\rrvert
\\
&&\qquad  =  \E\bigl[ \bigl\llvert f (\widebar{W}_n ) - f (
\bar{s}_n ) \bigr\rrvert\mathbf{1}_{E_n ( t, M
)} \bigr] + \E
\bigl[ \bigl\llvert f (\widebar{W}_n ) - f ( \bar{s}_n
) \bigr\rrvert( 1 - \mathbf{1}_{
E_n ( t, M )} ) \bigr]
\\
&&\qquad  \leq \E\bigl[ K \llVert\widebar{W}_n - \bar{s}_n
\rrVert\mathbf{1}_{E_n ( t, M )} \bigr] + \E\bigl[ 2 K ( 1 -
\mathbf{1}_{ E_n ( t, M )} ) \bigr]
\\
&&\qquad  \leq 2 K n^{-1/3} \E\bigl[ \# \bigl\{ \mbox{cycle half-edges
attached to $v_i$, $i \leq t n^{2/3}$} \bigr\}
\mathbf{1}_{E_n
( t, M )} \bigr]
\\
&&\quad\qquad{} + 2 K \mathbb{P} \Bigl( \max_{i \leq t} \Bigl\{
\bar{s}_n(i) - \min_{k \leq i}  \bar{s}_n(k) \Bigr\}
\geq M \Bigr).
\end{eqnarray*}
Lemma~\ref{formybertoi} ensures that
\[
\lim_{n \rightarrow\infty} n^{-1/3} \E\bigl[ \# \bigl\{ \mbox{cycle half-edges attached to $v_i$, $i \leq t n^{2/3}$}
\bigr\} \mathbf{1}_{E_n ( t, M )} \bigr] = 0. %
\]
Moreover, applying Corollary~\ref{jeanferr} and the Portmanteau theorem,
\[
\limsup_{n \rightarrow\infty} \mathbb{P} \Bigl( \max_{i \leq t }
\Bigl\{ \bar{s}_n(i) - \min_{k \leq i}
\bar{s}_n(k) \Bigr\} \geq M \Bigr) \leq\mathbb{P} \Bigl(
\max_{s \leq t }  \Bigl\{ W^{\nu}(s) - \min_{u \leq s} W^{\nu}(u) \Bigr\} \geq M \Bigr). %
\]
Therefore, for every $M > 0$,
\[
\limsup_{n \rightarrow\infty} \bigl\llvert\E\bigl[f (\widebar
{W}_n ) \bigr] - \E\bigl[f (\bar{s}_n ) \bigr]
\bigr\rrvert\leq2 K \mathbb{P} \Bigl( \max_{s \leq t }  \Bigl
\{
W^{\nu}(s) - \min_{u \leq s}    W^{\nu}(u)
\Bigr\} \geq M \Bigr). %
\]
Now, the continuity of $W^{\nu}$ implies that
\[
\lim_{M \rightarrow\infty} \mathbb{P} \Bigl( \max_{s \leq t }
\Bigl\{ W^{\nu}(s) - \min_{u \leq s}
W^{\nu}(u) \Bigr\} \geq M \Bigr) = 0. %
\]
Hence
\[
\lim_{n \rightarrow\infty} \E\bigl[f (\widebar{W}_n ) \bigr] -
\E\bigl[f (\bar{s}_n ) \bigr] =0. %
\]
Theorem~\ref{thetheoaldsboooo} is therefore proved.

\section{Study of the random simple graph $\mathcal{SG}_n$} \label{christinagoldlousub}

The setting of this section is the same as before: $\nu$ is supposed
to satisfy (\ref{revivalprepaa}). We intend to show Theorem~\ref
{jesuisallevoirfghfytu} (recall that, as before, Theorem~\ref
{lasttheofmylifeee} follows from Theorem~\ref{jesuisallevoirfghfytu}).
The proof is divided into two steps. First (see Lemma~\ref
{cestunlemmcleechrsit} below) we shall prove that, with probability
tending to 1, the possible loops and multiple edges in $\mathcal{G}_n$
arrive only after the first $\lfloor n^{3/4} \rfloor$ vertices have
been explored during the depth-first search. We shall then deduce that
the walk $\widebar{W}_n$ conditioned on the event $\{ \mathcal{G}_n$ is simple$\}$ has the same asymptotic behavior as
the walk $\widebar{W}_n$; see Section~\ref{cestladernisectdupap}.

\subsection{Time arrival of loops and multiple edges}

In this section, we prove the following result:
%
%
\begin{lemme} \label{cestunlemmcleechrsit}
Let $T(n)$ be the minimal index of a vertex of $\mathcal{G}_n$ having
a loop or a multiple edge:
\[
T(n) = \inf\bigl\{i \in\{1,\ldots, n\}\dvtx \mbox{$v_i$ has a loop or
a multiple edge} \bigr\}. %
\]
Then
\[
\lim_{t \rightarrow\infty} \mathbb{P}\bigl(T(n) > n^{3/4}\bigr) = 1.
\]
\end{lemme}
Observe that $T(n) = \infty$ if and only if $\mathcal{G}_n$ is simple.

\begin{pf*}{Proof of Lemma \ref{cestunlemmcleechrsit}}
Obviously, it suffices to show that
\[
\lim_{n \rightarrow\infty} \E\bigl[ \# \bigl\{ \mbox{loops or multiple
edges attached to $v_i$, $i \leq n^{3/4}$, in $\mathcal
{G}_n$} \bigr\} \bigr] = 0. %
\]
Let us establish this assertion. We proceed the same way as in the
proof of Lemma~\ref{formybertoi}. For every deterministic sequence
$(x_1,\ldots, x_n)$ of positive integers such that $\sum_{i=1}^n x_i$
is even, conditionally on the event $(\widehat{D}_1,\ldots, \widehat
{D}_n)=(x_1,\ldots, x_n)$, one has
\begin{eqnarray*}
&& \E\bigl[ \# \bigl\{\mbox{loops or multiple edges attached to
$v_i$, $i \leq n^{3/4} $} \bigr\} \llvert
\widehat{D}_1 = x_1,\ldots, \widehat{D}_n =
x_n \bigr]
\\
&&\qquad \leq \sum_{i=1}^{ n^{3/4}} \sum
_{k=1}^{x_i} \mathbb{P} \bigl( \mbox{the $k$th
half-edge of $v_i$ creates a loop or a multiple edge} \llvert
\\
&&\hspace*{225pt} \widehat{D}_1 = x_1,\ldots, \widehat{D}_n =
x_n \bigr).
\end{eqnarray*}
Now, the $k$th half-edge of a vertex with degree $x_i$ ($i \leq
n^{3/4}$) creates a loop with probability at most $\frac
{x_i-k}{n-n^{3/4}}$. It creates a multiple edge with probability at
most $\frac{k-1}{n-n^{3/4}}$. Consequently
\begin{eqnarray*}
&& \E\bigl[ \# \bigl\{\mbox{loops or multiple edges attached to
$v_i$, $i \leq n^{3/4} $} \bigr\} \llvert
\widehat{D}_1 = x_1,\ldots, \widehat{D}_n =
x_n \bigr]
\\
&&\qquad \leq \sum_{i=1}^{ n^{3/4} } \sum
_{k=1}^{x_i} \frac{x_i}{n-n^{3/4}}
\end{eqnarray*}
and
\begin{eqnarray*}
&& \E\bigl[ \# \bigl\{ \mbox{loops or multiple edges attached to
$v_i$, $i \leq n^{3/4}$, in $\mathcal{G}_n$}
\bigr\} \bigr]
\\
&&\qquad \leq \frac{1}{n-n^{3/4}} \E\Biggl[ \sum_{i=1}^{ n^{3/4}}
\widehat{D}_i^2 \Biggr]
\\
&&\qquad \leq \frac{n^{3/4}}{n-n^{3/4}} \E\bigl[ \widehat{D}_1^2 \bigr].
\end{eqnarray*}
Since $\nu$ has finite third moment, this completes the proof.\vadjust{\goodbreak}
\end{pf*}

\subsection{End of the proof of Theorem~\texorpdfstring{\protect\ref{jesuisallevoirfghfytu}}{3.2}} \label{cestladernisectdupap}

Let $t > 0$. Let $f$ be a bounded, continuous function defined on $(\D
([0,t]), \mathbb{R})$. It suffices to prove that
\[
\lim_{n \rightarrow\infty} \mathbb{E} \bigl[ f (
\widebar{W}_n ) \rrvert T(n) = \infty\bigr] = \mathbb{E} \bigl[ f
\bigl( {W}^\nu\bigr) \bigr], %
\]
where $f( \widebar{W}_n )$ denotes the image of the restriction of
the walk $\widebar{W}_n $ to $[0, t]$ by $f$. Let us show that
result. Observe that the event
\[
\bigl\{ \mbox{neither loop nor multiple edge is attached to $v_i$,
$i > n^{3/4}$} \bigr\} %
\]
is asymptotically independent of the r.v. $f ( \widebar{W}_n )$
\begin{eqnarray*}
\hspace*{-4pt}&&\mathbb{E} \bigl[ f ( \widebar{W}_n ) \mathbf{1}_{\mathrm{neither\
loop\ nor\ multiple\ edge\ is\ attached\ to}\ v_i, i >
n^{3/4}}
\bigr]
\\
\hspace*{-4pt}&&\quad \undersett{n \rightarrow\infty} {\sim} \mathbb{E} \bigl[ f (
\widebar{W}_n ) \bigr] \mathbb{P} \bigl(\mbox{neither loop nor
multiple edge is attached to $v_i$, $i > n^{3/4}$} \bigr).
\end{eqnarray*}
Deciding whether or not a loop or a multiple edge is created after step
$n^{3/4}$ does indeed not depend on the first $t n^{2/3}$
steps.\footnote{To make this argument rigorous, consider the
Poissonian model introduced in Section~\ref{poivguyiaeuy};
independence is then straightforward, and the fact that the two models
are asymptotically close has already been seen.}

Now, by Theorem~\ref{thetheoaldsboooo},
\[
\lim_{n \rightarrow\infty} \mathbb{E} \bigl[ f ( \widebar{W}_n
) \bigr] = \mathbb{E} \bigl[ f \bigl( {W}^\nu\bigr) \bigr].
\]
Moreover,
\fontsize{10pt}{\baselineskip}\selectfont
\begin{eqnarray*}
0 & \leq& \frac{\mathbb{P} (\mbox{neither loop nor
multiple edge is attached to $v_i$, $i > n^{3/4}$} )}{\mathbb
{P} ( T(n) = \infty)} - 1
\\
& = & \frac{\mathbb{P} ( \mbox{neither loop nor multiple
edge is attached to $v_i$, $i > n^{3/4}$ and $T(n) \leq n^{3/4}$}
)}{ \mathbb{P} ( \mbox{$\mathcal{G}_n$ is a simple
graph} )}
\\
& \leq& \frac{\mathbb{P} ( T(n) \leq n^{3/4} )}{
\mathbb{P} ( \mathcal{G}_n \mbox{ is a simple graph})}.
\end{eqnarray*}
\normalsize
According to Lemma~\ref{cestunlemmcleechrsit} and equation~(\ref
{derniequhabyvu}),
\[
\lim_{n \rightarrow\infty} \frac{\mathbb{P} ( T(n) \leq
n^{3/4} )}{ \mathbb{P} ( \mathcal{G}_n \mbox{ is a
simple graph} )} = 0. %
\]
Consequently,
\fontsize{10pt}{\baselineskip}\selectfont
\[
\mathbb{P} \bigl( \mbox{neither loop nor multiple edge is attached to
$v_i$, $i > n^{3/4}$} \bigr) \undersett{n \rightarrow\infty}
{\sim} \mathbb{P} \bigl( T(n) = \infty\bigr). %
\]
\normalsize

We finally obtain
\begin{eqnarray*}
&& \mathbb{E} \bigl[ f ( \widebar{W}_n ) \mathbf{1}_{\mathrm{neither\ loop\
nor\ multiple\ edge\ is\ attached\ to}\ v_i, i > n^{3/4}}
\bigr]
\\
&&\qquad \undersett{n \rightarrow\infty} {\sim} \mathbb{E} \bigl[ f \bigl(
{W}^\nu\bigr) \bigr] \mathbb{P} \bigl( T(n) = \infty\bigr).
\end{eqnarray*}
Recalling Lemma~\ref{cestunlemmcleechrsit} and equation~(\ref
{derniequhabyvu}) again (just proceed as before), this proves that
\[
\mathbb{E} \bigl[ f ( \widebar{W}_n ) \mathbf{1}_{T(n)=\infty}
\bigr] \undersett{n \rightarrow\infty} {\sim} \mathbb{E} \bigl[ f \bigl(
{W}^\nu\bigr) \bigr] \mathbb{P} \bigl( T(n) = \infty\bigr),
\]
completing the proof of Theorem~\ref{jesuisallevoirfghfytu}.

\section{The power law distribution setting} \label{powerlawsecr}

In this section, we do not suppose the finiteness of the moment of
order 3 for distribution $\nu$, and rather we replace assumption~(\ref
{revivalprepaa}) by
%
%
\begin{equation}
\label{thevrairevivalprepaa} \sum_{k=1}^{\infty} k (k-2)
\nu_k =0 \quad\mbox{and}\quad\nu_k \undersett{k
\rightarrow\infty} {\sim} c k^{-\gamma},
\end{equation}
where $c>0$ and $\gamma\in(3,4)$.
This implies that (\ref{equdebertoiiii}) has to be replaced by
%
%
\begin{eqnarray}
\label{equdebertoiiiidess} \mathcal{L}''(t) & = & 2 \mu-
\frac{c   \Gamma(4-\gamma
)}{\gamma-3} t^{\gamma-3} + \undersett{t \rightarrow0} {o}
\bigl(t^{\gamma-3}\bigr).
\end{eqnarray}
We are interested in the component sizes of the multigraph constructed
the same way as before. To have a good idea of what the order of the
component sizes should be, we adopt the same strategy, using Poisson
calculus; see Section~\ref{poissoniagetyv}. We shall then show in
Section~\ref{mainrespoelow} how to deduce the asymptotic behavior of
the component sizes of $\mathcal{G}_n$ in our new situation. In
Section~\ref{openprobsec} we shall state some open problems.

\subsection{The Poissonian argument} \label{poissoniagetyv}

Taking the same notation as in Section~\ref{sectkiuyy}, we here
consider the process $(S_n(t))_{t \geq0}$ defined by
\begin{eqnarray*}
S_n(t) & = & \sum_{(s,k) \in\Pi_n} (k-2)
\mathbf{1}_{s \leq t}.
\end{eqnarray*}
Recall that $\Pi_n$ is a Poisson point process on $(0,n) \times
\mathbb{N}^{\ast}$ with intensity $\pi_n$, where
$
\pi_n(\mathrm{d} t, k) = \nu_k k e^{-k\psi(t/n)} \psi'(t/n)\,\mathrm{d} t$.
We intend to prove the following result:
%
%
\begin{theoreme} \label{monpetitbbht}
Rescale $S_n$ by defining $
\widebar{S}_n(t) = n^{-1/(\gamma-1)} S_n (t n^{(\gamma-
2)/(\gamma- 1)} )$. Then
\begin{eqnarray*}
\widebar{S}_n & \mathop{\longrightarrow}\limits
_{n \rightarrow
\infty}^{(d)} &
X^{\nu}+A^{\nu}
\end{eqnarray*}
with respect to the Skorokhod topology on every finite interval,
where
\[
A^{\nu}_t = - \frac{c   \Gamma(4-\gamma)}{(\gamma
-3)(\gamma- 2) \mu^{\gamma-2}} t^{\gamma-2},\qquad t
\geq0, %
\]
and $X^{\nu}$ is the unique process with independent increments such
that for every $t \geq0$ and $u \in\mathbb{R}$,
\begin{eqnarray*}
\mathbb{E} \bigl[ \exp\bigl( i u X^{\nu}_t \bigr) \bigr] &
= & \exp\biggl( \int_0^t \mathrm{d}s \int
_0^\infty\mathrm{d}x \bigl( e^{i u x} -1 -
iux \bigr) \frac{c}{\mu} \frac{1}{x^{\gamma-1}} e^{-x s / \mu} \biggr).
\end{eqnarray*}
\end{theoreme}
\begin{pf}
As before, let
\[
A_n(t) = \int_{(0,n) \times\mathbb{N}^{\ast}} \pi_n (
\mathrm{d}s, k) (k-2) \mathbf{1}_{s \leq t},\qquad t \geq0 %
\]
be the deterministic continuous bounded variation function such that
\[
M_n (t) = S_n (t) - A_n (t),\qquad t \geq0
\]
is a martingale. Just as we rescaled $S_n$ to form $\widebar{S}_n$ in
Theorem~\ref{monpetitbbht}, write $\widebar{A}_n$ and $\widebar
{M}_n$ for the correspondingly rescaled versions of $A_n$ and $M_n$.
Note that (\ref{equdebertoiiii}) was the only ingredient of the proof
of Lemma~\ref{babyonemofgt}. Since in our setting equation (\ref
{equdebertoiiiidess}) holds, we can perform the same elementary
calculations and then find that for every $t > 0$,
\[
\lim_{n \rightarrow\infty} \sup_{s \leq t} \bigl\llvert{
\widebar{A}}_n(s) -A^{\nu}_s \bigr\rrvert= 0.
\]
To complete the proof of Theorem~\ref{monpetitbbht}, it thus suffices
to show that
\[
\widebar{M}_n \mathop{\longrightarrow}\limits
_{n \rightarrow
\infty}^{(d)}
X^{\nu} %
\]
with respect to the Skorokhod topology on every finite interval.
We postpone the proof of that result to the \hyperref[app]{Appendix}; see Lemmas~\ref
{babyonemofgtjjudeu} and \ref{babyonemofgtjjutros}.
\end{pf}

\subsection{The main result} \label{mainrespoelow}

Repeating exactly what we did in Section~\ref{ghyttteerfeiii}, we
deduce from Theorem~\ref{monpetitbbht} the following key result. As
before, the walk defined via (\ref{defindelamarcle}) is denoted by $W_n$.
%
%
\begin{coro} \label{hehoilfatleciterq}
Rescale the depth-first walk $W_n$ by defining for every $t \in[0,
n^{1/(\gamma-1)}]$,
\begin{eqnarray*}
\widebar{W}_n(t) & = & n^{-1/(\gamma-1)} W_n \bigl(
\bigl\lfloor t n^{(\gamma-2)/(\gamma-1)} \bigr\rfloor\bigr).
\end{eqnarray*}
Then
\begin{eqnarray*}
\widebar{W}_n & \mathop{\longrightarrow}\limits
_{n \rightarrow
\infty}^{(d)} &
X^{\nu}+A^{\nu}
\end{eqnarray*}
with respect to the Skorokhod topology on every finite interval.
\end{coro}

We now give an analogous result of Theorem~\ref{keytheorr} in the
present setting. Let $R^{\nu}$ be the reflected process defined by
\[
R^{\nu}_t = X^{\nu}_t+A^{\nu}_t
- \inf_{0 \leq s \leq t} \bigl\{ X^{\nu}_s+A^{\nu}_s
\bigr\},\qquad t \geq0. %
\]
We define excursion intervals and excursion lengths of $R^{\nu}$ as in
Section~\ref{sectionfrretwo}.
%
%
\begin{theoreme} \label{levraikeytheorr}
Suppose $\nu$ satisfies (\ref{thevrairevivalprepaa}). Then a.s.
the set of excursions of $R^{\nu}$ may be written $\{ \gamma_j, j
\geq1 \}$ so that the lengths $|\gamma_j|$ are decreasing and
\[
\sum_{j \geq1} |\gamma_j|^2 <
\infty, %
\]
and letting $\bolds{\mathcal{C}}^{\nu}_n$ be the ordered
sequence of component sizes of $\mathcal{G}_n$,
\begin{eqnarray*}
n^{-(\gamma-2)/(\gamma-1)} \bolds{\mathcal{C}}^{\nu}_n &
\mathop{\longrightarrow}\limits
_{n \rightarrow\infty}^{(d)} & \bigl( \llvert
\gamma_j \rrvert, j \geq1 \bigr)
\end{eqnarray*}
with respect to the $l_{\searrow}^2$ topology.
\end{theoreme}

Contrary to the finite third moment case, Theorem~\ref
{levraikeytheorr} cannot been seen as a straightforward consequence of
Corollary~\ref{hehoilfatleciterq}; the analogy of Section~3.4 of \cite
{MR1434128} does not exist here. The following lemma (which uses
Corollary~\ref{hehoilfatleciterq}) will nonetheless enable us to get
Theorem~\ref{levraikeytheorr}. We refer to Section~\ref{deuxpart} for
the definitions of $\zeta(n, k)$ and $\mathcal{C}(n,k)$.

%
%
\begin{lemme} \label{encoreunle}
For every positive integer $n$, let $\Xi^{(n)}$ be the point process
\[
\Xi^{(n)} = \bigl\{ \bigl( n^{-(\gamma-2)/(\gamma-1)} \zeta(n,k-1),
n^{-(\gamma-2)/(\gamma-1)} \mathcal{C}(n,k) \bigr)\dvtx k \geq1 \bigr\},
\]
and let $\Xi^{(\infty)}$ be the point process
\[
\Xi^{(\infty)} = \bigl\{ \bigl(l(\gamma), |\gamma|\bigr)\dvtx \gamma
\mbox{ is an excursion of } R^\nu\bigr\}. %
\]
Then $\Xi^{(n)}$ converges vaguely in distribution to $\Xi^{(\infty
)}$ as $n \rightarrow\infty$.\footnote{Vague convergence of counting
measures on $[0, \infty) \times(0, \infty)$ is considered here.}
Moreover, $\Xi^{(\infty)}$ satisfies the following three points:
\begin{longlist}[(3)]
\item[(1)] $\sup{}\{ s\dvtx (s,y) \in\Xi^{(\infty)} \mbox{ for
some } y \} = \infty$ a.s.;
\item[(2)] if $(s,y) \in\Xi^{(\infty)}$, then $\sum_{(s',y') \in
\Xi^{(\infty)}\dvtx s'<s} y' = s$ a.s.;
\item[(3)] $\max{}\{ y\dvtx (s,y) \in\Xi^{(\infty)} \mbox{ for
some } s > s_0 \} \stackrel{p}{\rightarrow} 0$ as $s_0 \rightarrow
\infty$.
\end{longlist}
\end{lemme}
\begin{pf}
Observe that the component sizes of the multigraph $\mathcal{G}_n$, in
the order of appearance in depth-first walk, are size-biased ordered.
Following the \emph{proof} of~\cite{MR1491528}, Proposition~17,
Lemma~\ref{encoreunle} thus derives from Corollary~\ref
{hehoilfatleciterq} and forthcoming Lemma~\ref{lederdesderheeh} stated
in the \hyperref[app]{Appendix}.
\end{pf}

\begin{pf*}{Proof of Theorem~\ref{levraikeytheorr}}
Applying \cite{MR1491528}, Proposition~17 (see also \cite{MR1434128},
Proposition~15 and Lemma~25), Lemma~\ref{encoreunle} ensures that
a.s. the set of excursions of $R^{\nu}$ can be written $\{ \gamma
_j, j \geq1 \}$ so that the lengths $|\gamma_j|$ are decreasing and
\[
\sum_{j \geq1} |\gamma_j|^2 <
\infty. %
\]
By \cite{MR1491528}, Proposition~17, another consequence of Lemma~\ref
{encoreunle} is that
\begin{eqnarray*}
n^{-(\gamma-2)/(\gamma-1)} \bolds{\mathcal{C}}^{\nu}_n &
\mathop{\longrightarrow}\limits
_{n \rightarrow\infty}^{(d)} & \bigl( \llvert
\gamma_j \rrvert, j \geq1 \bigr)
\end{eqnarray*}
with respect to the $l_{\searrow}^2$ topology.
\end{pf*}

\subsection{Open questions} \label{openprobsec}

The argument used to prove Lemma~\ref{cestunlemmcleechrsit} does not
work in our present setting. Observe though that (\ref
{derniequhabyvu}) still holds here. That is why we believe that the
following result is true:
%
%
\begin{conjecture} \label{ceciestmanecbei}
Suppose $\nu$ satisfies (\ref{thevrairevivalprepaa}).
Let $\bolds{\mathcal{SC}}^{\nu}_n$ be the ordered sequence of
component sizes of $\mathcal{SG}_n$ and $( \llvert\gamma_j \rrvert,
j \geq1 )$ be the ordered sequence of the excursion lengths of $R^\nu
$. Then
\begin{eqnarray*}
n^{-(\gamma-2)/(\gamma-1)} \bolds{\mathcal{SC}}^{\nu}_n &
\mathop{\longrightarrow}\limits
_{n \rightarrow\infty}^{(d)} & \bigl( \llvert
\gamma_j \rrvert, j \geq1 \bigr)
\end{eqnarray*}
with respect to the $l_{\searrow}^2$ topology.
\end{conjecture}

As before, Conjecture~\ref{ceciestmanecbei} would be proven as soon as
the following result is shown:
%
%
\begin{conjecture}
If $\nu$ satisfies (\ref{thevrairevivalprepaa}), then the rescaled
walk $\widebar{W}_n$ conditioned on the event $\{ \mathcal
{G}_n$ is simple$\}$ converges in distribution to $X^{\nu}+A^{\nu}$
with respect to the Skorokhod topology on every finite interval as $n
\rightarrow\infty$.
\end{conjecture}

\setcounter{equation}{0}
\begin{appendix}\label{app}
\section{The finite third moment setting}

In this section, we complete the proof of Proposition~\ref
{gilfoallalllavernewww} by showing two technical results.

%
%
\begin{lemme} \label{babyonemofgt}
For every $t_0 > 0$,
\[
\lim_{n \rightarrow\infty} \sup_{t \leq t_0} \biggl\llvert{
\widebar{A}}_n(t) +\frac{\beta}{2 \mu^2} t^2 \biggr\rrvert=
0. %
\]
\end{lemme}
\begin{pf}
By definition,
\begin{eqnarray*}
A_n(t) & = & \int_0^{t} \sum
_{k \in\N^{\ast}} \bigl(k^2 - 2 k\bigr)
e^{-k\psi(s/n)} \psi'(s/n) \nu_k \,\mathrm{d} s
\\
& = & \int_0^t \bigl(a_n(s) -2
\bigr) \,\mathrm{d}s,
\end{eqnarray*}
where
\[
a_n(s) = \frac{\mathcal{L}''(\psi(s/n))}{- \mathcal{L}'(\psi(s/n))}. %
\]
Since $\E[D^2] = 2 \E[D]$, $a_n(s)$ tends to 2 as $n \rightarrow
\infty$. Moreover, it is easily seen by approximating $\psi(s/n)$ by
$\frac{s}{\mu n}$ that $a_n(s)-2$ is approximatively $-\frac{\beta}{
\mu^2} \frac{s}{n}$. Let us be more precise. Recalling (\ref
{equdebertoiiii}), in the neighborhood of $t=0$,
\[
\mathcal{L}''(t) = 2 \mu- (\beta+ 4 \mu) t + o(t)
\quad\mbox{and}\quad\mathcal{L}'(t) = - \mu+2 \mu t + o(t).
\]
Therefore
\[
\frac{\mathcal{L}''(t) + 2 \mathcal{L}'(t)}{- \mathcal{L}'(t)} = -\frac
{\beta}{\mu} t + o(t), %
\]
that is, there exists a function $\varepsilon^{(1)}(\cdot)$ tending
to 0 at 0 such that
\[
\frac{\mathcal{L}''(t)}{- \mathcal{L}'(t)} - 2 = -\frac{\beta}{\mu
} t + t \varepsilon^{(1)}(t).
\]
Now, $\lb(t) = \frac{t}{\mu} + o(t)$ so that there exists a function
$\varepsilon^{(2)}(\cdot)$ tending to 0 at 0 such that
\[
\psi(t) = \frac{t}{\mu} + t \varepsilon^{(2)}(t). %
\]
We deduce that
\begin{eqnarray*}
a_n(s) - 2 & = & -\frac{\beta}{\mu} \psi\biggl( \frac{s}{n}
\biggr) + \psi\biggl( \frac{s}{n} \biggr) \varepsilon^{(1)}
\biggl( \psi\biggl( \frac{s}{n} \biggr) \biggr)
\\
& = & -\frac{\beta}{\mu} \biggl( \frac{s}{\mu n} + \frac{s}{n}
\varepsilon^{(2)} \biggl( \frac{s}{n} \biggr) \biggr)
\\
&&{}+ \biggl(
\frac{s}{\mu n} + \frac{s}{n} \varepsilon^{(2)} \biggl(
\frac{s}{n} \biggr) \biggr) \varepsilon^{(1)} \biggl(
\frac{s}{\mu n} + \frac
{s}{n} \varepsilon^{(2)} \biggl(
\frac{s}{n} \biggr) \biggr)
\\
& = & -\frac{\beta}{\mu^2} \frac{s}{ n} + \frac{s}{ n} \biggl\{ -
\frac{\beta}{\mu} \varepsilon^{(2)} \biggl( \frac{s}{n} \biggr) +
\frac{1}{\mu} \varepsilon^{(1)} \biggl( \frac{s}{\mu n} +
\frac
{s}{n} \varepsilon^{(2)} \biggl( \frac{s}{n} \biggr)
\biggr)
\\
&&\hspace*{87pt}{} + \varepsilon^{(2)} \biggl( \frac{s}{n} \biggr)
\varepsilon^{(1)} \biggl( \frac{s}{\mu n} + \frac{s}{n}
\varepsilon^{(2)} \biggl( \frac{s}{n} \biggr) \biggr) \biggr\}.
\end{eqnarray*}
Defining
\[
\varepsilon\dvtx t \mapsto- \frac{\beta}{\mu} \varepsilon^{(2)} ( t ) +
\frac{1}{\mu} \varepsilon^{(1)} \biggl( \frac{t}{\mu} + t
\varepsilon^{(2)} ( t ) \biggr) + \varepsilon^{(2)} ( t )
\varepsilon^{(1)} \biggl( \frac{t}{\mu} + t \varepsilon^{(2)}
( t ) \biggr), %
\]
we finally get
\[
a_n(s) - 2= -\frac{\beta}{\mu^2} \frac{s}{n} +
\frac
{s}{n}\varepsilon\biggl( \frac{s}{n} \biggr) %
\]
with $\varepsilon(\cdot)$ tending to 0 at 0. Thus, for every $t \in
[0, t_0 n^{2/3}]$,
\[
\biggl\llvert A_n(t) + \frac{\beta}{ \mu^2} \frac{t^2}{2 n} \biggr
\rrvert\leq\frac{1}{n} \int_0^{t} s
\biggl\llvert\varepsilon\biggl( \frac{s}{n} \biggr) \biggr\rrvert
\,\mathrm{d}s \leq\frac{1}{n} \int_0^{t_0 n^{2/3}} s
\biggl\llvert\varepsilon\biggl( \frac{s}{n} \biggr) \biggr\rrvert\,
\mathrm{d}s. %
\]
As a result, for every $\eta> 0$, there exists an integer $n_0(\eta)$
such that for every integer $n \geq n_0(\eta)$,
\[
\sup_{t \leq t_0 n^{2/3}} \biggl\llvert A_n(t) +
\frac{\beta}{2 \mu^2} \frac{t^2}{n} \biggr\rrvert\leq\frac{1}{n} \int
_0^{t_0 n^{2/3}} s \eta\,\mathrm{d}s = \frac{t_0^2}{2}
\eta n^{1/3}, %
\]
which proves Lemma~\ref{babyonemofgt}.
\end{pf}

%
%
\begin{lemme} \label{babyonemofgtjju}
$\widebar{M}_n \mathop{\longrightarrow}\limits_{n
\rightarrow\infty}^{(d)} \sqrt{\frac{\beta}{ \mu}} B$ with
respect to the Skorokhod topology on every finite interval, where $B$
denotes a standard Brownian motion.
\end{lemme}
\begin{pf}
We want to apply the functional CLT for continuous-time martingales.
Since $A_n$ is continuous, and $S_n$ only jumps at points $t_j$,
$M_n$ is a purely discontinuous martingale, so that
$
[ M_n ]_t = \sum_{s \leq t} \Delta M_n(s)^2
$
and its predictable projection
\[
\langle M_n \rangle(t) = \int_{(0,n) \times\mathbb{N}^{\ast}} \pi
_n (\mathrm{d}s, k) (k-2)^2 \mathbf{1}_{s \leq t},\qquad t \geq0, %
\]
is the continuous, increasing process such that $M_n^2 - \langle M_n
\rangle$ is a martingale. Observe that $\langle M_n \rangle$ is
deterministic. Define $\langle\widebar{M}_n \rangle(t) = n^{-2/3}
\langle M_n \rangle(t n^{2/3})$. Applying \cite{MR838085}, Theorem~7.1.4(b), the following two points imply Lemma~\ref
{babyonemofgtjju}:
%
%
\begin{equation}
\label{condalddeux} \forall t_0 > 0\qquad{\langle
\widebar{M}_n \rangle}(t_0) \undersett{n \rightarrow\infty} {
\longrightarrow} \frac{\beta}{ \mu} t_0
\end{equation}
and
%
%
\begin{equation}
\label{condaldtrois} \lim_{n \rightarrow\infty} \E\Bigl[ \sup_{t \leq t_0}
\bigl\llvert\widebar{M}_n(t) - \widebar{M}_n(t-) \bigr
\rrvert^2 \Bigr] = 0.
\end{equation}

Let us establish (\ref{condalddeux}). First note that
\begin{eqnarray*}
\langle M_n \rangle(t) & = & \int_0^{t}
\sum_{k \in\N^{\ast}} k(k - 2)^2 e^{-k\psi(s/n)}
\psi'(s/n) \nu_k \,\mathrm{d} s
\\
& = & \int_0^t b_n(s)\,
\mathrm{d}s,
\end{eqnarray*}
where
\[
b_n(s) = \frac{\mathcal{L}^{(3)} + 4\mathcal{L}'' + 4\mathcal
{L}'}{\mathcal{L}'} \circ\psi\biggl( \frac{s}{n}
\biggr). %
\]
Since $\psi(t)$ tends to 0 as $t \rightarrow0$ and
\[
\lim_{t \rightarrow0} \frac{\mathcal{L}^{(3)}(t) + 4\mathcal
{L}''(t) + 4\mathcal{L}'(t)}{\mathcal{L}'(t)} = \frac{-(\beta+4\mu
)+ 8 \mu- 4 \mu}{-\mu}=
\frac{\beta}{\mu}, %
\]
there exists a function $\varepsilon(\cdot)$ tending to 0 at 0 such that
\[
b_n(s) = \frac{\beta}{ \mu} + \varepsilon\biggl( \frac{s}{n}
\biggr). %
\]
We deduce that
\[
\biggl\llvert\langle M_n \rangle\bigl(t_0
n^{2/3}\bigr) - \frac{\beta}{ \mu} t_0 n^{2/3}
\biggr\rrvert\leq\int_0^{t_0 n^{2/3}} \biggl\llvert
\varepsilon\biggl( \frac{s}{n} \biggr) \biggr\rrvert\,\mathrm{d}s.
\]
Hence, for every $\eta> 0$, there exists an integer $n_1(\eta)$ such
that for every integer $n \geq n_1(\eta)$,
\[
\biggl\llvert\langle M_n \rangle\bigl(t_0
n^{2/3}\bigr) - \frac{\beta}{ \mu} t_0 n^{2/3}
\biggr\rrvert\leq\eta t_0 n^{2/3}, %
\]
which proves (\ref{condalddeux}).

We next turn our attention to (\ref{condaldtrois}). Note that
$M_n(t)-M_n(t-) = S_n(t)-S_n(t-)$, so
\begin{eqnarray*}
\sup_{t \leq t_0 n^{2/3}} \bigl|M_n(t) - M_n(t-)\bigr|^2
& = & \sup\bigl\{ (k-2)^2\dvtx (s,k) \in\Pi_n\mbox{ and
} s \leq t_0 n^{2/3} \bigr\}
\\
& \leq&\sup\bigl\{ k^2\dvtx (s,k) \in\Pi_n\mbox{ and }
s \leq t_0 n^{2/3} \bigr\}.
\end{eqnarray*}
Let $L_n$ denote $\sup\{ k\dvtx (s,k) \in\Pi_n$ and $s \leq
t_0 n^{2/3} \}$ (we drop the dependency on $t_0$ in the notation). We have
\[
\E\bigl[L_n^2 \bigr] = \sum
_{k=1}^{\lfloor n^{1/3} \rfloor-1} \mathbb{P} ( L_n \geq\sqrt{k}
) + \sum_{k \geq
n^{1/3}} \mathbb{P} ( L_n \geq
\sqrt{k} ) \leq n^{1/3} + \sum_{k \geq n^{1/3}}
\mathbb{P} ( L_n \geq\sqrt{k} ). %
\]
Now, for every $m \in\mathbb{N}$,
\begin{eqnarray*}
\mathbb{P} ( L_n \geq m ) & = & 1- \mathbb{P} ( L_n < m
)
\\
& = & 1- \mathbb{P} \bigl( \Pi_n \bigl( \bigl[0, t_0
n^{2/3}\bigr] \times\{ m, m+1,\ldots\} \bigr) =0 \bigr)
\\
& = & 1-\exp\bigl( -\pi_n \bigl( \bigl[0, t_0
n^{2/3}\bigr] \times\{ m, m+1,\ldots\} \bigr) \bigr)
\\
& \leq& \pi_n \bigl( \bigl[0, t_0 n^{2/3}
\bigr] \times\{ m, m+1,\ldots\} \bigr)
\\
& = & \sum_{l \geq m} \nu_l \int
_0^{ t_0 n^{2/3}} \mathrm{d}s\, l e^{- l \psi(s/n)}
\psi'(s/n)
\\
& = & n \sum_{l \geq m} \nu_l \bigl(
1-e^{- l \lb(t_0 n^{-1/3})} \bigr)
\\
& \leq& n \lb\bigl(t_0 n^{-1/3}\bigr) \sum
_{l \geq m} l \nu_l.
\end{eqnarray*}
As a result,
\begin{eqnarray*}
\E\bigl[L_n^2 \bigr] & \leq& n^{1/3} + \sum
_{k \geq n^{1/3}} n \lb\bigl(t_0 n^{-1/3}
\bigr) \sum_{l \geq\sqrt{k}} l \nu_l
\\
& = & n^{1/3} + n \lb\bigl(t_0 n^{-1/3}\bigr)
\sum_{l \geq n^{1/6}} l \nu_l \sum
_{k= \lfloor n^{1/3} \rfloor}^{l^2} 1.
\end{eqnarray*}
We deduce that for every integer $n$,
\[
n^{-2/3} \E\Bigl[ \sup_{t \leq t_0 n^{2/3}} \bigl|M_n(t) -
M_n(t-)\bigr|^2 \Bigr] \leq n^{-1/3} +
n^{1/3} \lb\bigl(t_0 n^{-1/3}\bigr) \sum
_{l \geq
n^{1/6}} l^3 \nu_l. %
\]
Now, $n^{1/3} \lb(t_0 n^{-1/3})$ tends to $\frac{t_0}{\mu}$ and
since $\E[D^3]$ is finite, $\sum_{l \geq n^{1/6}} l^3 \nu_l$ tends
to 0. Equation (\ref{condaldtrois}) is therefore proved.
\end{pf}

\section{The power law distribution setting}

\subsection{End of the proof of Theorem~\texorpdfstring{\protect\ref{monpetitbbht}}{8.1}}

This section is organized as follows. In Lemma~\ref
{babyonemofgtjjudeu} we shall study the martingale $\widebar
{M}_n^{(1)}$ related to the small jumps of~$\widebar{M}_n$. Then, in
Lemma~\ref{babyonemofgtjjutros}, we shall be interested in the
martingale $\widebar{M}_n^{(2)}$ which counts the big jumps. The fact
that $\widebar{M}_n = \widebar{M}_n^{(1)} + \widebar{M}_n^{(2)}$
converges to $X^{\nu}$, which is the sum of the limits of $\widebar
{M}_n^{(1)}$ and $\widebar{M}_n^{(2)}$, stems from the independence
of $\widebar{M}_n^{(1)}$ and $\widebar{M}_n^{(2)}$ (since they
never jump simultaneously).
To ease notation, let
\[
a = \frac{1}{\gamma-1}. %
\]

%
%
\begin{lemme} \label{babyonemofgtjjudeu}
The martingale $\widebar{M}_n^{(1)}$ defined for every $t \geq0$ by
\begin{eqnarray*}
\widebar{M}_n^{(1)} (t) &=& \sum
_{(s,k) \in\Pi_n} \mathbf{1}_{k <
n^{a}} (k-2)n^{-a}
\mathbf{1}_{s \leq t n^{1-a}}
\\
& &{} - \int_{ (0, n) \times\mathbb{N}^{\ast} } \pi_n (\mathrm{d}s, k)
\mathbf{1}_{k < n^{a}} (k-2) n^{-a} \mathbf{1}_{s \leq t n^{1-a}}
\end{eqnarray*}
converges in distribution with respect to the Skorokhod topology on
every finite interval as $n \rightarrow\infty$ to a process
$(X^{(1)}_t)_{t \geq0}$ with independent increments characterized by:
for every $t \geq0$ and $u \in\mathbb{R}$,
\[
\mathbb{E} \bigl[ \exp\bigl( i u X^{(1)}_t \bigr) \bigr] =
\exp\biggl( \int_0^t \mathrm{d}s \int
_0^1 \mathrm{d}x \bigl( e^{i u x} -1 -iux
\bigr) \frac{c}{\mu} \frac{1}{x^{\gamma-1}} e^{-x s / \mu} \biggr).
\]
\end{lemme}
\begin{pf}
First observe that the process $X^{(1)}$ may be defined as
the limit for the metric induced by the norm
\[
\| Y \| = \mathbb{E} \bigl[ \sup\bigl\{ Y_s^2\dvtx 0 \leq s
\leq t \bigr\} \bigr]^{1/2} %
\]
of the Cauchy family
\[
t \mapsto\sum_{s \leq t} \mathbf{1}_{ \Delta_s > \varepsilon}
\Delta_s - \int_0^t \mathrm{d}s \int
_\varepsilon^1 \mathrm{d}x\,   x \frac{c}{\mu}
\frac{1}{x^{\gamma-1}} e^{-xs/\mu} %
\]
as $\varepsilon$ tends to 0, $\Delta$ being a Poisson point process
with intensity $ \mathbf{1}_{ x \in(0,1)} \nu(\mathrm{d}s, {\mathrm
{d}}x)$ where
\[
\nu(\mathrm{d}s, \mathrm{d}x) = \frac{c}{\mu} \frac{1}{x^{\gamma-1}}
e^{-xs/\mu} \,\mathrm{d}s \,\mathrm{d}x. %
\]
To prove Lemma~\ref{babyonemofgtjjudeu}, we rely on \cite{MR1943877},
Theorem~VII.3.7. Dealing with small jumps of the martingale
$\widebar{M}_n$ indeed enables us to work with ``square-integrable''
processes [note that $\int_0^t \int_\mathbb{R} x^2 \mathbf{1}_{ x
\in(0,1)} \nu(\mathrm{d}s, \mathrm{d}x) < \infty$].

Taking the same notation as in \cite{MR1943877}, we first have to
compute the characterics $(B^n, C^n, \nu^n)$ of $\widebar
{M}_n^{(1)}$, which are defined via the following equation: for every
$t \geq0$ and $u \in\mathbb{R}$,
\begin{eqnarray*}
&& \mathbb{E} \bigl[ \exp\bigl( i u \widebar{M}_n^{(1)} (t)
\bigr) \bigr]
\\
&&\qquad = \exp\biggl( iuB^n(t) - \frac{1}{2}u^2C^n(t)
+ \int_0^t\! \int_{-n^{-a}}^1
\bigl( e^{i u x} -1 -iux \bigr) \nu^n( \mathrm{d}s, \mathrm{d}x )
\biggr).
\end{eqnarray*}
The exponential formula for Poisson point processes yields
\begin{eqnarray*}
&& \mathbb{E} \bigl[ \exp\bigl( i u \widebar{M}_n^{(1)} (t)
\bigr) \bigr]
\\
&&\qquad = \exp\biggl\{ n \sum_{k < n^a}
\nu_k \bigl( 1 -e^{-k \psi
(t n^{-a})} \bigr) \bigl( e^{iu (k-2)n^{-a}} -1 -
iu (k-2)n^{-a} \bigr) \biggr\}.
\end{eqnarray*}
Consequently, $B^n =C^n = 0$ and
\[
\nu^n ( \mathrm{d}s, \mathrm{d}x ) = \mathrm{d}s \sum
_{k < n^a} \delta_{(k-2)n^{-a}} (\mathrm{d}x )
n^{1-a} k \nu_k \psi'\bigl(s
n^{-a}\bigr) e^{-k \psi(s n^{-a})}. %
\]
According to \cite{MR1943877}, Theorem~VII.3.7, Lemma~\ref
{babyonemofgtjjudeu} will be proved as soon as we have shown that for
every $t \geq0$,
%
%
\begin{eqnarray}
\label{peemheqaverjacoo} \int_0^t \!\int
_{-n^{-a}}^1 x^2 \nu^n( {
\mathrm{d}}s, \mathrm{d}x ) & \undersett{n \rightarrow\infty}
{\longrightarrow} & \int
_0^t \!\int_{0}^1
x^2 \nu( \mathrm{d}s, \mathrm{d}x ),
\end{eqnarray}
and for every $t \geq0$ and $g \in C_2(\mathbb{R_+})$,
%
%
\begin{eqnarray}
\label{deuxxheqaverjacoo} \int_0^t \!\int
_{-n^{-a}}^1 g(x) \nu^n( \mathrm{d}s, {
\mathrm{d}}x ) & \undersett{n \rightarrow\infty} {\longrightarrow} & \int
_0^t \!\int_{0}^1
g(x) \nu( \mathrm{d}s, \mathrm{d}x ),
\end{eqnarray}
where $ C_2(\mathbb{R_+})$ is the set of all continuous bounded
functions $\mathbb{R_+} \to\mathbb{R}$ which are~0 on a neighborhood
0 and have a limit at infinity.

Let us establish (\ref{peemheqaverjacoo}). Elementary calculations yield
\begin{eqnarray*}
\int_0^t \!\int_{-n^{-a}}^1
x^2 \nu^n( \mathrm{d}s, \mathrm{d}x ) & = & n^{1-2a}
\sum_{k<n^a} (k-2)^2 \nu_k
\bigl(1-e^{-k \psi(t n^{-a})} \bigr).
\end{eqnarray*}
A difficulty stems from the lack of good estimates for $\nu_k$ when
$k$ is small. That is why we write
\begin{eqnarray*}
\int_0^t \!\int_{-n^{-a}}^1
x^2 \nu^n( \mathrm{d}s, \mathrm{d}x ) & = & n^{1-2a}
\sum_{k \in\mathbb{N}^*} (k-2)^2 \nu_k
\bigl(1-e^{-k \psi
(t n^{-a})} \bigr)
\\
& &{} - n^{1-2a} \sum_{k \geq n^a}
(k-2)^2 \nu_k \bigl(1-e^{-k \psi(t
n^{-a})} \bigr).
\end{eqnarray*}
It is easy to see that the first term in the difference tends to $
\frac{c \Gamma(4-\gamma)}{(\gamma-3) \mu^{\gamma-3}} t ^{\gamma
-3}$. As for the second, recalling that $\nu_k \sim c k^{-\gamma}$,
\begin{eqnarray*}
&& n^{1-2a} \sum_{k \geq n^a} (k-2)^2
\nu_k \bigl(1-e^{-k \psi(t
n^{-a})} \bigr)
\\
&&\qquad  \undersett{n \rightarrow\infty}
{\sim} n^{1-2a} \int_{n^a}^{\infty} \mathrm{d}x\,
x^2 c x^{-\gamma} \bigl(1-e^{-x \psi
(t n^{-a})} \bigr).
\end{eqnarray*}
A change of variable and an application of the dominated convergence
theorem [recall that $\lb(x) = \frac{x}{\mu} + o(x)$] yield
\[
n^{1-2a} \sum_{k \geq n^a} (k-2)^2
\nu_k \bigl(1-e^{-k \psi(t
n^{-a})} \bigr) \undersett{n \rightarrow\infty}
{\longrightarrow} \int_{1}^{\infty} \mathrm{d}x\,
c \frac{1-e^{-x t / \mu
}}{x^{\gamma-2}}.
\]
Noticing that
\[
\frac{c   \Gamma(4-\gamma)}{(\gamma-3)\mu^{\gamma-3}} t ^{\gamma
-3} = \int_{0}^{\infty}
\mathrm{d}x \,  c \frac
{1-e^{-x t / \mu}}{x^{\gamma-2}}, %
\]
and we finally get
\begin{eqnarray*}
\int_0^t \!\int_{-n^{-a}}^1
x^2 \nu^n( \mathrm{d}s, \mathrm{d}x ) & \undersett{n
\rightarrow\infty} {\longrightarrow} & \int_{0}^{1}
\mathrm{d}x \,  c \frac{1-e^{-x t / \mu}}{x^{\gamma-2}},
\end{eqnarray*}
which proves (\ref{peemheqaverjacoo}).

We now turn our attention to (\ref{deuxxheqaverjacoo}). Let
$\varepsilon\in(0,1)$ and $g\dvtx [\varepsilon, 1] \to\mathbb{R}$ be
a continuous function. Then
\begin{eqnarray*}
\int_0^t \!\int_{-n^{-a}}^1
g(x) \nu^n( \mathrm{d}s, \mathrm{d}x ) & = & n \sum
_{\varepsilon n^a < k<n^a} g \biggl(\frac{k-2}{n^a} \biggr) \nu_k
\bigl(1-e^{-k \psi(t n^{-a})} \bigr).
\end{eqnarray*}
Proceeding as before, we obtain
\begin{eqnarray*}
\int_0^t \!\int_{-n^{-a}}^1
g(x) \nu^n( \mathrm{d}s, \mathrm{d}x ) & \undersett{n \rightarrow
\infty} {
\longrightarrow} & \int_{\varepsilon}^{1} \mathrm{d}x\, g(x) c
\frac{1-e^{-x t / \mu
}}{x^{\gamma}},
\end{eqnarray*}
completing the proof of Lemma~\ref{babyonemofgtjjudeu}.
\end{pf}

In order to finish the proof Theorem~\ref{monpetitbbht}, we now
show the convergence of the martingale related to the big jumps.
%
%
\begin{lemme} \label{babyonemofgtjjutros}
The martingale $\widebar{M}_n^{(2)}$ defined for every $t \geq0$ by
\begin{eqnarray*}
\widebar{M}_n^{(2)} (t) &=& \sum
_{(s,k) \in\Pi_n} \mathbf{1}_{k
\geq n^{a}} (k-2)n^{-a}
\mathbf{1}_{s \leq t n^{1-a}}
\\
& &{} - \int_{(0, n) \times\mathbb{N}^{\ast}} \pi_n (\mathrm{d}s,k)
\mathbf{1}_{k \geq n^{a}} (k-2) n^{-a} \mathbf{1}_{s \leq t n^{1-a}}
\end{eqnarray*}
converges in distribution with respect to the Skorokhod topology on
every finite interval as $n \rightarrow\infty$ to a process
$(X^{(2)}_t)_{t \geq0}$ with independent increments characterized by:
for every $s, t \geq0$, $u \in\mathbb{R}$,
\[
\mathbb{E} \bigl[ \exp\bigl( i u X^{(2)}_t \bigr) \bigr] =
\exp\biggl( \int_0^t \mathrm{d}s \int
_1^\infty\mathrm{d}x \bigl( e^{i u
x} -1 -iux
\bigr)\frac{c}{\mu} \frac{1}{x^{\gamma-1}} e^{-xs/\mu
} \biggr). %
\]
\end{lemme}
\begin{pf}
The existence of $X^{(2)}$ is easily obtained as the sum of
%
%
\begin{equation}
\label{alorslakoiyt} B^{\nu}_t = - \int_0^t
\mathrm{d}s \int_1^\infty\mathrm{d}x \,  x
\frac{c}{\mu} \frac{1}{x^{\gamma-1}} e^{-xs/\mu},\qquad t \geq0,
\end{equation}
and the partial sum of the jumps of a Poisson point process with
intensity $\mathbf{1}_{ x \geq1} \nu(\mathrm{d}s, \mathrm{d}x) $
[recall that
$
\nu(\mathrm{d}s, \mathrm{d}x) = \frac{c}{\mu} \frac{1}{x^{\gamma-1}}
e^{-xs/\mu}\, \mathrm{d}s\, \mathrm{d}x$].
Let\vspace*{2pt} us see how Lemma~\ref{babyonemofgtjjutros} derives from \cite
{MR1943877}, Theorem~VII.3.4.

As before, we first have to compute the characterics $(B^n, C^n, \nu
^n)$ of $\widebar{M}_n^{(2)}$, which are now defined via the
equation: for every $s, t \geq0$, $u \in\mathbb{R}$,
\begin{eqnarray*}
&& \mathbb{E} \bigl[ \exp\bigl( i u \widebar{M}_n^{(2)} (t)
\bigr) \bigr]
\\
&&\qquad = \exp\biggl( iu B^n(t) - \frac{1}{2}u^2
C^n(t) + \int_0^t \!\int
_{1- 2 n^{-a}}^{\infty} \bigl( e^{i u x} -1 \bigr)
\nu^n( \mathrm{d}s, \mathrm{d}x ) \biggr).
\end{eqnarray*}
The exponential formula for Poisson point processes yields
\begin{eqnarray*}
\mathbb{E} \bigl[ \exp\bigl( i u \widebar{M}_n^{(2)} (t)
\bigr) \bigr] & = & \exp\biggl\{ -iu n^{1-a} \sum
_{k \geq n^a} (k-2) \nu_k \bigl(1 -e^{-k \psi(t n^{-a})}
\bigr)
\\
&&\hspace*{20pt}{} + n \sum_{k \geq n^a} \nu_k
\bigl(1-e^{-k \psi(t
n^{-a})} \bigr) \bigl( e^{iu (k-2)n^{-a}} -1 \bigr) \biggr\}.
\end{eqnarray*}
Consequently, $C^n = 0$,
\[
B^n(t) = - n^{1-a} \sum_{k \geq n^a}
(k-2) \nu_k \bigl( 1 -e^{-k
\psi(t n^{-a})} \bigr) %
\]
and
\[
\nu^n ( \mathrm{d}s, \mathrm{d}x ) = \mathrm{d}s \sum
_{k \geq n^a} \delta_{(k-2)n^{-a}} (\mathrm{d}x )   k
\nu_k n^{1-a} \psi'\bigl(s n^{-a}
\bigr) e^{-k \psi(s n^{-a})}. %
\]
According to \cite{MR1943877}, Theorem~VII.3.4, Lemma~\ref
{babyonemofgtjjutros} will be proved as soon as we have shown that for
every $t \geq0$,
%
%
\begin{eqnarray}
\label{peemheqaverjacaa} \sup_{s \leq t} \bigl\llvert B^n(t)-B^{\nu}_t
\bigr\rrvert& \undersett{n \rightarrow\infty} {\longrightarrow} & 0,
\end{eqnarray}
and for every $t \geq0$ and $g \in C_2(\mathbb{R_+})$,
%
%
\begin{eqnarray}
\label{deuxxheqaverjacaa} \int_0^t \!\int
_{1- 2 n^{-a}}^{\infty} g(x) \nu^n( \mathrm{d}s, \mathrm{d}x ) &
\undersett{n \rightarrow\infty} {\longrightarrow} & \int
_0^t \!\int_1^{\infty}
g(x) \nu( \mathrm{d}s, \mathrm{d}x ).
\end{eqnarray}
Equation (\ref{deuxxheqaverjacaa}) can be shown exactly the same way
as (\ref{deuxxheqaverjacoo}), and to prove (\ref{peemheqaverjacaa}),
it suffices to compare the series to the corresponding integrals as we
did above.
\end{pf}

\subsection{Completion of the proof of Lemma~\texorpdfstring{\protect\ref{encoreunle}}{8.4}}

In this section, we give the missing elements in the proof of
Lemma~\ref{encoreunle}. This is provided by Lemma~\ref{lederdesderheeh}.

%
%
\begin{lemme} \label{lederdesderheeh}
The following four assertions hold:
\begin{enumerate}
\item[(1)] \label{premierpoint} $ X^{\nu}_t + A^{\nu}_t \stackrel
{p}{\rightarrow} - \infty$ as $t \rightarrow\infty$;
\item[(2)] \label{premierpointtro} $\sup{}\{|\gamma|\dvtx \gamma$
is an excursion of $R^{\nu}$ s.t.  $l(\gamma) \geq t \}
\stackrel{p}{\rightarrow} 0$ as $t \rightarrow\infty$;\vadjust{\goodbreak}
\item[(3)] \label{premierpointqua} The set $\{ t\dvtx R^{\nu}_t = 0 \}$
contains no isolated points a.s.;
\item[(4)] \label{premierpointde} For every $t > 0$, $\mathbb
{P}(R^{\nu}_t=0)=0$.
\end{enumerate}
\end{lemme}

\begin{pf*}{Proof of Lemma \ref{premierpoint}(1)}
By Lemma~\ref{babyonemofgtjjudeu},
\[
\mathbb{E} \bigl[ \bigl( X^{(1)}_t \bigr)^2
\bigr] = \frac
{c}{\mu} \int_0^t \mathrm{d}s
\int_0^1 \mathrm{d}x \frac{1}{x^{\gamma
-3}}
e^{-x s / \mu} \leq c t \int_0^{1/t} \mathrm{d}x
\frac
{1}{x^{\gamma-3}} + c \int_{1/t}^\infty\mathrm{d}x
\frac
{1}{x^{\gamma-2}},
\]
so that
\[
\mathbb{E} \bigl[ \bigl( X^{(1)}_t \bigr)^2
\bigr] \leq\frac
{c}{(\gamma-3)(4-\gamma)} t^{\gamma-3}.
\]
Applying Markov's inequality, we deduce that
%
%
\begin{eqnarray}
\label{equatclepourpointtr} t^{-(\gamma-3)} X^{(1)}_t & \mathop{
\longrightarrow}\limits
_{t
\rightarrow\infty}^{p}& 0.
\end{eqnarray}
Letting $\eta= (\gamma-3)/2$, this implies that
$t^{-(1+\eta)} X^{(1)}_t \stackrel{p}{\rightarrow} 0$ as $t
\rightarrow\infty$.
Then notice that $X^{(2)}_t$ is less than $\sum_{s \leq t} \Delta_s
$, where $\Delta$ is a Poisson point process with intensity $ \mathbf
{1}_{ x \geq1} \nu(\mathrm{d}s, \mathrm{d}x)$ [recall that $\nu
(\mathrm{d}s, \mathrm{d}x) = \frac{c}{\mu} \frac{1}{x^{\gamma-1}}
e^{-xs/\mu}\, \mathrm{d}s\, \mathrm{d}x
$]. Now
$
\mathbb{E}[\sum_{s \leq t} \Delta_s] = \frac{c}{\mu} \int_0^t
\mathrm{d}s \int_1^\infty\mathrm{d}x \frac{1}{x^{\gamma-2}} e^{-x s /
\mu} \leq\frac{c}{\mu(\gamma-3)} t$.
Consequently, by\break  Markov's inequality, $t^{-(1+\eta)} \sum_{s \leq t}
\Delta_s \stackrel{p}{\rightarrow} 0$ as $t \rightarrow\infty$.
Since $t^{-(1+\eta)} A^{\nu}_t \rightarrow- \infty$ as $t
\rightarrow\infty$, property~(1) is proved.
\end{pf*}

\begin{pf*}{Proof of Lemma \ref{premierpointtro}(2)}
Restate (2) as follows: for every $\varepsilon> 0$,
\[
\mbox{number of } \bigl(\mbox{excursion of }R^\nu\mbox{ with length }>
2\varepsilon\bigr) < \infty\qquad\mbox{a.s.}
\]
Fix\vspace*{2pt} $\varepsilon> 0$ and define events $C_n = \{ \sup_{s \in[(n-1)
\varepsilon, n \varepsilon]} ( X^\nu_{(n+1)\varepsilon} + A^\nu
_{(n+1)\varepsilon} - X^\nu_s - A^\nu_s ) > 0\}$. It is easily seen
that it suffices to show that $\mathbb{P}(C_n$ infinitely
often$)=0$. By (\ref{equatclepourpointtr}), it is enough to prove that
%
%
\begin{equation}
\label{oufffffdisocnf} \sum_{n \geq1+ s_0/\varepsilon} \mathbb{P}
\bigl(C_n \cap C^{s_0}\bigr) < \infty\qquad\mbox{for every
large $s_0$},
\end{equation}
where\vspace*{1pt} $C^{s_0} = \{ \sup_{t \geq s_0} t^{-(\gamma-3)} |X^{(1)}_t|
\leq\delta\}$ for some positive (small) constant $\delta> 0$ to be
chosen later. Now
\begin{eqnarray*}
C_n & \subset&\biggl\{  \sup_{s \in[(n-1) \varepsilon, n
\varepsilon]} \bigl(
X^{(2)}_{(n+1)\varepsilon} - X^{(2)}_s \bigr)
\\
&&\hspace*{6pt} \geq
\frac{c   \Gamma(4-\gamma)}{(\gamma-3)(\gamma- 2)
\mu^{\gamma-2}} \varepsilon^{\gamma-2} \bigl((n+1)^{\gamma-2} -
n^{\gamma-2} \bigr)
\\
&&\hspace*{87pt}{} - \sup_{s \in[(n-1) \varepsilon, n \varepsilon]} \bigl(
X^{(1)}_{(n+1)\varepsilon} -
X^{(1)}_s \bigr) \biggr\}.
\end{eqnarray*}
For every $n$ larger than $1+ s_0/\varepsilon$, on $C^{s_0}$, we have
\begin{eqnarray*}
\sup_{s \in[(n-1) \varepsilon, n \varepsilon]} \bigl(
X^{(1)}_{(n+1)\varepsilon} -
X^{(1)}_s \bigr) & \leq& 2 \delta\varepsilon^{\gamma-3}
(n+1)^{\gamma-3} \leq2 \delta\varepsilon^{\gamma-3} 2^{\gamma-3}
n^{\gamma-3}.
\end{eqnarray*}
Consequently, for every $n$ larger than $1+ s_0/\varepsilon$,
\begin{eqnarray*}
C_n &\cap& C^{s_0} \subset\biggl\{ \sup_{s \in[(n-1) \varepsilon, n
\varepsilon]}
\bigl( X^{(2)}_{(n+1)\varepsilon} - X^{(2)}_s \bigr)
\\
&&\hspace*{35pt} \geq\biggl( \frac{c   \Gamma(4-\gamma)}{(\gamma-3) \mu
^{\gamma-2}} \varepsilon^{\gamma-2} - \delta
\varepsilon^{\gamma
-3} 2^{\gamma-2} \biggr) n^{\gamma-3} \biggr\}.
\end{eqnarray*}
Taking
$
\delta= \varepsilon\frac{c   \Gamma(4-\gamma)}{(\gamma-3)
\mu^{\gamma-2} 2^{\gamma-1}}$,
and denoting
$
\frac{c   \Gamma(4-\gamma)}{2 (\gamma-3) \mu^{\gamma-2}}
\varepsilon^{\gamma-2}
$
by $\rho$,
we thus have for every $n$ large enough,
\[
C_n \cap C^{s_0} \subset\Bigl\{ \sup_{s \in[(n-1) \varepsilon, n
\varepsilon]}
\bigl( X^{(2)}_{(n+1)\varepsilon} - X^{(2)}_s \bigr)
\geq\rho n^{\gamma-3} \Bigr\}. %
\]
Now,\vspace*{1pt} considering a Poisson point process $\Delta$ with intensity $
\mathbf{1}_{ x \geq1} \nu(\mathrm{d}s, \mathrm{d}x)$, where $\nu
(\mathrm{d}s, \mathrm{d}x) = \frac{c}{\mu} \frac{1}{x^{\gamma-1}}
e^{-xs/\mu}\, \mathrm{d}s\, \mathrm{d}x
$, observe that
\begin{eqnarray*}
& & \mathbb{P} \Bigl( \sup_{s \in[(n-1) \varepsilon, n \varepsilon
]} \bigl( X^{(2)}_{(n+1)\varepsilon}
- X^{(2)}_s \bigr) \geq\rho n^{\gamma-3} \Bigr)
\\
&&\qquad \leq \mathbb{P} \biggl( \sum_{s \in[(n-1)\varepsilon,
(n+1)\varepsilon]}
\Delta_s \geq\rho n^{\gamma-3} \biggr)
\\
&&\qquad  \leq \rho^{-1} n^{-\gamma+3} \mathbb{E} \biggl[ \sum
_{s \in
[(n-1)\varepsilon, (n+1)\varepsilon]} \Delta_s \biggr]
\\
&&\qquad  =  \rho^{-1} n^{-\gamma+3} \frac{c}{\mu} \int
_{(n-1)\varepsilon
}^{(n+1)\varepsilon} \mathrm{d}s \int_1^{\infty}
\mathrm{d}x \,  x \frac{1}{x^{\gamma-1}} e^{-xs/\mu}.
\end{eqnarray*}
We deduce that for every $n$ larger than $2+ s_0/\varepsilon$,
\begin{eqnarray*}
\mathbb{P} \bigl( C_n \cap C^{s_0} \bigr)& \leq&
\frac{2
\varepsilon c}{\rho\mu} n^{-\gamma+3} \int_1^{\infty} {
\mathrm{d}}x  \, x^{2-\gamma} e^{-n x \varepsilon/(2 \mu)} \leq
\frac{4
c}{\rho}
n^{-\gamma+2} e^{-n \varepsilon/(2 \mu)},
\end{eqnarray*}
which proves (\ref{oufffffdisocnf}) and completes the proof of
assertion~(2).
\end{pf*}

\begin{pf*}{Proof of Lemma \ref{premierpointqua}(3)}
To show property~(3), we first consider the case
$t=0$. We aim at showing that
$
\inf\{ s > 0\dvtx X^{\nu}_s + A^{\nu}_s < 0 \} = 0$ a.s.
To do so, we shall in fact prove an analogous result for a certain L\'
evy process, which will be obtained by using standard properties of L\'
evy processes. We shall deduce property~(3) by
comparing our process $X^\nu$ with the studied L\'evy process.

Observe\vspace*{1pt} that for every $s \in[0,\infty)$ and $x \in(0, \infty)$,
$\frac{c}{\mu} \frac{1}{x^{\gamma-1}} e^{-x s / \mu} \leq\frac
{c}{\mu} \frac{1}{x^{\gamma-1}}$. Recalling the two remarks\vspace*{1pt} situated
at the beginning of the proofs of Lemmas~\ref{babyonemofgtjjudeu}~and~\ref{babyonemofgtjjutros} (we described there a way to define
$X^{(1)}$ and $X^{(2)}$), we can couple the process $X^{\nu}$ and
construct a stable process $L$ with index $\gamma-2$ with no negative
jumps such that
\[
\forall s \geq0, \forall u \in\mathbb{R} \qquad\mathbb{E} \bigl[
\exp( i u
L_s ) \bigr] = \exp\biggl( s \int_0^\infty{
\mathrm{d}}x \bigl( e^{i u x} -1 -iux \bigr) \frac{c}{\mu}
\frac{1}{x^{\gamma-1}} \biggr)
\]
satisfying
\begin{eqnarray*}
\forall s \geq0\qquad X^{\nu}_s & \leq& L_s +
\frac{c}{\mu} \int_0^s \mathrm{d}r \int
_0^\infty\mathrm{d}x \frac{1}{x^{\gamma-2}} \bigl( 1-
e^{-x r / \mu} \bigr),
\end{eqnarray*}
that is,
\begin{eqnarray*}
\forall s \geq0\qquad X^{\nu}_s & \leq& L_s +
\frac{c
\Gamma(4-\gamma)}{(\gamma-3)(\gamma- 2) \mu^{\gamma-2}} s^{\gamma-2}.
\end{eqnarray*}
Consequently
$
X^\nu+ A^{\nu} \leq L$.
Since
$
\inf\{ s > 0\dvtx L_s < 0 \} = 0$ a.s.,
with probability 1, 0 is not an isolated point of the set $\{t\dvtx
R^\nu
_t=0 \}$.

This is now standard to get assertion~(3); see, for
instance, \cite{MR1406564}, Proposition~VI.4, or the end of the proof
of assertion (d) of \cite{MR1491528}, Proposition~14.
\end{pf*}

\begin{pf*}{Proof of Lemma~\ref{premierpointde}(4)}
Here again, we shall use a coupling argument. Indeed, imagine we are
able to prove that for a certain process $(Q_s)_{s \in[0,t]}$,
\[
\mathbb{P} \bigl( Q_{t} = \inf\bigl\{ Q_s\dvtx s \in[0,t]
\bigr\} \bigr) = 0,
\]
and
for every $s \in[0, t]$,
\begin{eqnarray*}
X^{\nu}_{t} + A^{\nu}_{t} - \bigl(
X^{\nu}_{s} + A^{\nu}_{s} \bigr) &
\geq& Q_{t} - Q_{s}.
\end{eqnarray*}
Then, with probability one,
\[
\sup\bigl\{ X^{\nu}_{t} + A^{\nu}_{t} -
\bigl( X^{\nu}_{s} + A^{\nu}_{s} \bigr)\dvtx s \in[0, t] \bigr\} \geq
\sup\bigl\{ Q_{t} - Q_{s}\dvtx s
\in[0, t] \bigr\} > 0, %
\]
establishing assertion~(4). Let us prove that such a
coupling exists.

We have to bound the increments of $X^{\nu}+A^{\nu} $ from below. We
first focus on $X^{(1)}$. Let $t \in(0, \infty)$. Arguing as before
(just recall the remark made at the beginning of Lemma~\ref
{babyonemofgtjjudeu}), since for every $s \in[0,t]$ and $x \in(0,
1)$, $\frac{c}{\mu} \frac{1}{x^{\gamma-1}} e^{-x s / \mu} \geq
\frac{c}{\mu} \frac{1}{x^{\gamma-1}} e^{-x t / \mu} $, we can
construct a L\'{e}vy process $(Q^{(1)}_s)_{ s \in[0,t]}$ such that
%
\begin{eqnarray}
\mathbb{E} \bigl
[ \exp
\bigl( i u Q^{(1)}_s \bigr) \bigr] = \exp\biggl( s \int
_0^1 \mathrm{d}x \bigl( e^{i u x} -1 -iux
\bigr) \frac{c}{\mu
} \frac{1}{x^{\gamma-1}} e^{-x t / \mu} \biggr)\nonumber
\\
\eqntext{\forall s \in[0,t], \forall u \in\mathbb{R} }
\end{eqnarray}
satisfying
%
\begin{eqnarray}
X^{(1)}_{t} - X^{(1)}_{s}
\geq Q^{(1)}_{t} - Q^{(1)}_{s} +
\frac{c}{\mu} \int_s^{t} \mathrm{d}r \int
_0^1 \mathrm{d}x \frac{1}{x^{\gamma-2}} \bigl(
e^{-x t / \mu}- e^{-x r / \mu} \bigr)\nonumber
\\
\eqntext{\forall s \in[0,t].}
\end{eqnarray}
Since for every $a, b \in(0, \infty)$ such that $a<b$, $e^{-a}-e^{-b}
\leq b-a$, we have for every $s \in[0, t]$
\begin{eqnarray*}
X^{(1)}_{t} - X^{(1)}_{s} & \geq&
Q^{(1)}_{t} - Q^{(1)}_{s} -
\frac
{c}{2 (4-\gamma)\mu^2} (t - s)^2.
\end{eqnarray*}
Recalling the definition of $B^\nu$ [see (\ref{alorslakoiyt})], we
deduce that for every $s \in[0, t]$,
\begin{eqnarray*}
X^{\nu}_{t} - X^{\nu}_{s} & \geq&
Q^{(1)}_{t} - Q^{(1)}_{s} -
\frac
{c}{2 (4-\gamma)\mu^2} (t - s)^2 + B^{\nu}_t -
B^{\nu}_s.
\end{eqnarray*}
We easily deduce that there exists $C > 0$ (only depending on $t$) such
that for every $s \in[0, t]$,
\begin{eqnarray*}
X^{\nu}_{t} + A^{\nu}_{t} - \bigl(
X^{\nu}_{s} + A^{\nu}_{s} \bigr) &
\geq& Q^{(1)}_{t} - Q^{(1)}_{s} - C(t-s).
\end{eqnarray*}
Consequently,
\begin{eqnarray*}
&& \sup\bigl\{ X^{\nu}_{t} + A^{\nu}_{t} -
\bigl( X^{\nu}_{s} + A^{\nu}_{s} \bigr)\dvtx s \in[0, t] \bigr\}
\\
&&\qquad \geq
\sup\bigl\{ Q^{(1)}_{t} - C t -
\bigl( Q^{(1)}_{s} - Cs \bigr)\dvtx s \in[0, t] \bigr\}.
\end{eqnarray*}
Now, applying \cite{MR1406564}, Theorem~VII.2 and page 158, to the L\'
evy process $(Q^{(1)}_s - Cs)_{ s \in[0,t]}$, we have
\[
\mathbb{P} \bigl( Q^{(1)}_{t} - C t = \inf\bigl\{
Q^{(1)}_s - Cs\dvtx s \in[0,t] \bigr\} \bigr) = 0.
\]
We deduce that
\[
\mathbb{P} \bigl( X^{\nu}_{t} + A^{\nu}_{t}
= \inf\bigl\{ X^{\nu
}_{s} + A^{\nu}_{s}\dvtx s \in[0,t] \bigr\} \bigr) = 0,
\]
which is assertion~(4).
\end{pf*}
\end{appendix}

\section*{Acknowledgments}
I am grateful to Jean Bertoin who suggested this subject to me. I
warmly thank him for fruitful discussions and comments and for his
reading. I also thank Christina Goldschmidt and Louigi Addario-Berry
for help on extending the results to simple graphs, Svante Janson for
interesting comments and an anonymous referee for careful reading and
for very helpful suggestions.



%

\printaddresses

\end{document}